\documentclass[11pt]{amsart}
\usepackage{graphicx}
\usepackage{amssymb}
\usepackage{xcolor}
\usepackage{hyperref}
\date{}

\newtheorem{theorem}{Theorem}[section]
\newtheorem{lemma}{Lemma}[section]
\newtheorem{corollary}{Corollary}[section]
\newtheorem{proposition}{Proposition}[section]
\theoremstyle{definition}
\newtheorem{definition}{Definition}[section]
\newtheorem*{ex}{Example}
\newtheorem*{remark}{Remark}

\renewcommand{\S}{\mathbb{S}}

\newcommand{\D}{\mathbb{D}}
\newcommand{\N}{\mathbb{N}}
\newcommand{\Z}{\mathbb{Z}}

\newcommand{\R}{\mathbb{R}}
\newcommand{\C}{\mathbb{C}}

\newcommand{\La}{\Lambda}
\newcommand{\la}{\lambda}

\newcommand{\lr}{\rightarrow}

\newcommand{\sgn}{\operatorname{sgn}}
\newcommand{\Aut}{\operatorname{Aut}}

\newcommand{\Aug}{\operatorname{Aug}}
\newcommand{\Spec}{\operatorname{Spec}}

\newcommand{\st}{\text{st}}

\begin{document}
	\title{Lagrangian Fillings in A-type and Their K\'alm\'an Loop Orbits}
	
	\author{James Hughes}
\address{University of California Davis, Dept. of Mathematics, Shields Avenue, Davis, CA 95616, USA}
\email{jmhughes@math.ucdavis.edu}

	\begin{abstract}
	
	We compare two constructions of exact Lagrangian fillings of Legendrian positive braid closures, the Legendrian weaves of Casals-Zaslow, and the decomposable Lagrangian fillings, of Ekholm-Honda-K\'alm\'an and show that they coincide for large families of Lagrangian fillings. As a corollary, we obtain an explicit correspondence between Hamiltonian isotopy classes of decomposable Lagrangian fillings of Legendrian $(2,n)$ torus links described by Ekholm-Honda-K\'alm\'an and the weave fillings constructed by Treumann and Zaslow. We apply this result to describe the orbital structure of the K\'alm\'an loop and give a combinatorial criteria to determine the orbit size of a filling. We follow our geometric discussion with a Floer-theoretic proof of the orbital structure, where an identity studied by Euler in the context of continued fractions makes a surprise appearance. We conclude by giving a purely combinatorial description of the K\'alm\'an loop action on the fillings discussed above in terms of edge flips of triangulations. 
	

	\end{abstract}
	
	\makeatletter
\@namedef{subjclassname@2020}{%
  \textup{2020} Mathematics Subject Classification}
\makeatother



	\maketitle
	
		\section{Introduction}

Legendrian links and their exact Lagrangian fillings are objects of interest in contact and symplectic topology \cite{BourgeoisSabloffTraynor15, Chantraine10, EliashbergPolterovich96, EtnyreNg19}. Within the last decade, parallel developments in the constructive methods \cite{CasalsZaslow, EHK} and the application of both Floer-theoretic \cite{CasalsNg, GSW, YuPan} and microlocal-sheaf-theoretic \cite{CasalsGao, STZ_ConstrSheaves, TreumannZaslow} invariants have significantly advanced the classification of Lagrangian fillings. 
Broadly speaking, this manuscript aims to compare the two primary methods of constructing Lagrangian fillings, \cite{CasalsZaslow} and \cite{EHK}, of Legendrian positive braid closures. We then leverage this result in the well-studied case of Lagrangian fillings of Legendrian $(2, n)$ torus links in the standard contact 3-sphere to understand the action of a Legendrian loop on these fillings. In addition to a contact geometric approach, we discuss insights into properties of the augmentation variety associated to this class of Legendrian links afforded by this comparison. 

Denote by $\sigma_i$ the $i$th Artin generator of the $n$-stranded braid group $Br_n$ and let $\beta$ be a positive braid, i.e. a product of positive powers of $\sigma_i$. We define the family of maximal-tb Legendrian links $\la(\beta)$ in the front projection as the rainbow closure of $\beta$, as depicted in Figure \ref{fig: BraidClosure} (left). Legendrian $(2, n)$ torus links are given by the braid $\sigma_1^n \subseteq Br_2$. They are also smoothly described as the link of the complex $A_{n-1}$-singularity $f(x, y)= x^{n}+y^2$, and hence we will also denote Legendrian $(2, n)$ torus links by $\la(A_{n-1})$. The Lagrangian fillings that we will consider in this manuscript are all exact, orientable, and embedded in the standard symplectic 4-ball $(\D^4, \lambda_{st})$, whose boundary is the standard contact 3-sphere $(\S^3, \xi_{st})$.

	 	\begin{center}
		\begin{figure}[h!]{ \includegraphics[width=.8\textwidth]{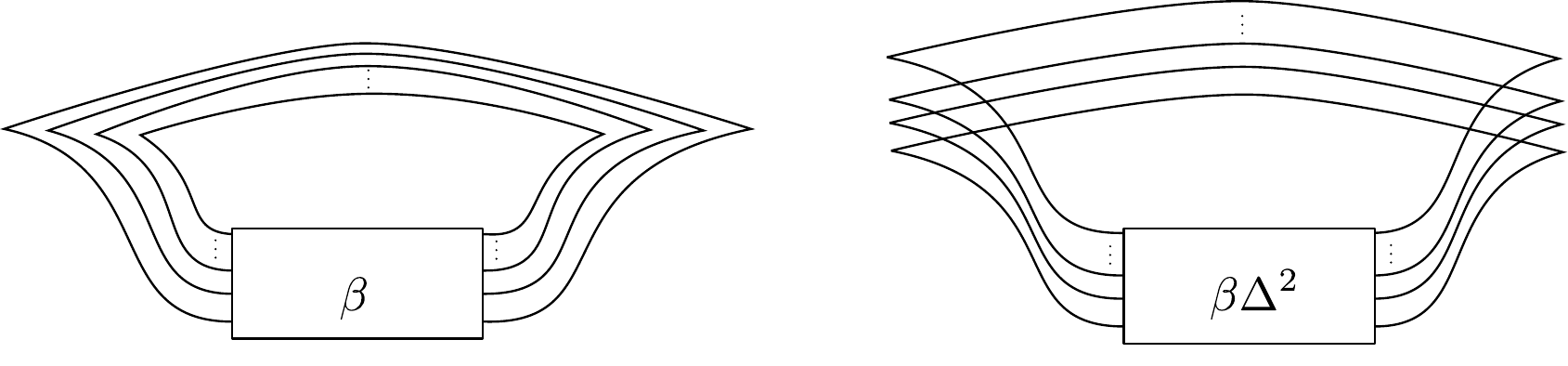}}\caption{Front projections of the Legendrian isotopic links given as the rainbow closure (left) and $(-1)$-framed closure (right) of the positive braids $\beta$ and $\beta\Delta^2$. Here $\Delta$ denotes a half twist of the braid.}
			\label{fig: BraidClosure}\end{figure}
	\end{center}

\subsection{Equivalence of Lagrangian fillings} This article has two primary independent contributions, one geometric and the other algebraic. We begin by sketching the two constructions of Lagrangian fillings needed to state our geometric result. The first general method of constructing Lagrangian fillings was given by Ekholm-Honda-K\'alm\'an in \cite{EHK}. In their construction, a filling $L$ of $\la(\beta)$ is given by a series of elementary cobordisms consisting of (1) traces of Legendrian isotopies, (2) pinching cobordisms, which can be seen as 0-resolutions of certain crossings, and (3) minimum cobordisms. In the case of $\la(A_{n-1}),$ the authors construct a family of $L_\sigma$ Lagrangian fillings using only pinching and minimum cobordisms, and label them by the permutation $\sigma$ in $S_n$ specifying an order of resolving the $n$ crossings of $\la(A_{n-1})$. The Lagrangian fillings constructed in \cite{EHK} were then separated into distinct Hamiltonian isotopy classes and indexed to 312-avoiding permutations by Yu Pan in \cite{YuPan}, expanding on the Floer-theoretic methods of \cite{EHK}. 312-avoiding permutations are permutations $\sigma$ in $S_n$ avoiding the appearance of numbers $i\, j\, k$ with $i> k>j$ when written in one-line notation. There are a Catalan number $C_{n}=\frac{1}{n+1}{2n\choose n}$ of 312-avoiding permutations in $S_n$, hence Pan's result shows that there are at least $C_{n}$ distinct Lagrangian fillings of $\la(A_{n-1})$ obtained via pinching cobordisms. These fillings will be referred to as pinching sequence fillings and the set of (Hamiltonian isotopy classes of) fillings as $\mathcal{P}_n$.

    In \cite{TreumannZaslow}, Treumann and Zaslow gave an alternative construction of a Catalan number of Lagrangian fillings of $\la(A_{n-1})$ and distinguished them using the microlocal theory of sheaves. These Lagrangian fillings are represented by planar trivalent graphs and indexed by the triangulated $(n+2)$-gons dual to such graphs. Given a triangulation $\mathcal{T}$ of a regular $(n+2)$-gon, we will denote the filling represented by the 2-graph dual to $\mathcal{T}$ by $L_{\mathcal{T}}$. 
    Casals and Zaslow then generalized the construction of \cite{TreumannZaslow} to the setting of positive braid closures in \cite{CasalsZaslow} with their construction of Legendrian weaves. We refer to Lagrangian fillings arising from this construction as weave fillings and denote the set of (Hamiltonian isotopy classes of) weave fillings of $\la(A_{n-1})$ as $\mathcal{W}_n$. We also refer to the elementary cobordism in this construction as a $D_4^-$ cobordism after Arnold's classification of singularities of fronts \cite{ArnoldSing}. 
    In Section \ref{sec: prelims} we give an explicit description of both a pinching cobordism and a $D_4^-$ cobordism. In Section \ref{section:geometric}, we produce a Hamiltonian isotopy between the local model describing an elementary pinching cobordism and the local model describing the $D_4^-$ cobordism. This equivalence is then used to prove our first main result:
 

	\begin{theorem}\label{thm: isotopy1}
		For any exact Lagrangian filling of	$\la(\beta)$ constructed via a sequence of pinching cobordisms and traces of Reidemeister III moves, there is unique a Hamiltonian isotopic weave filling. 

	
\end{theorem}

An immediate consequence of Theorem \ref{thm: isotopy1} is that the two sets of a Catalan number of (Hamiltonian isotopy classes of) exact Lagrangian fillings $\mathcal{W}_n$ and $\mathcal{P}_n$ constructed in \cite{TreumannZaslow} and \cite{EHK} coincide. This confirms statements appearing without proof in \cite[Section 2.3]{TreumannZaslow} and in \cite[Section 6]{STWZ} to that effect. The correspondence is also in agreement with the conjectured ADE-classification of exact Lagrangian fillings \cite[Conjecture 5.1]{CasalsLagSkel} which predicts precisely $C_n$ distinct fillings of $\la(A_{n-1})$ up to Hamiltonian isotopy. 
Combinatorially, Theorem \ref{thm: isotopy1} implies that for any 312-avoiding permutation $\sigma$, the filling $L_{\sigma}$ is Hamiltonian isotopic to the filling $L_{\mathcal{T}}$ for a unique triangulation $\mathcal{T}$. In Subsection \ref{Sub: combinatorics}, we give an explicit recipe for relating $\sigma$ and the corresponding triangulation $\mathcal{T}_\sigma$ for Hamiltonian isotopic fillings of $\la(A_{n-1})$ based on a combinatorial bijection of \cite{RegevAlon2013Abbt}. 


\subsection{K\'alm\'an loop orbits}

Theorem \ref{thm: isotopy1} appears as a protagonist in another central narrative of our study, an exploration of the orbital structure of the K\'alm\'an loop action on Lagrangian fillings of $\la(A_{n-1})$. Introduced by the eponymous mathematician in \cite{Kalman}, the K\'alm\'an loop is a Legendrian isotopy that acts on the set of fillings of a Legendrian torus link by permuting the order in which crossings are resolved by elementary cobordism. For weave fillings, the K\'alm\'an loop action is readily understood by the combinatorics of the triangulation of the dual $(n+2)$-gon under the action of rotation. Theorem \ref{thm: isotopy1} therefore allows us to geometrically deduce the orbital structure of the K\'alm\'an loop action on the set of pinching sequence fillings where it is otherwise more mysterious.

Second, independently of our geometric result, we also give a Floer-theoretic proof of the orbital structure of the K\'alm\'an loop by examining its action on the augmentation variety $\Aug(\la(A_{n-1}))$. The augmentation variety is a Floer-theoretic invariant associated to a Legendrian link. In \cite{EHK}, it was shown that a filling of a Legendrian $\la$ endowed with a choice of a local system can be interpreted geometrically as a point in the augmentation variety $\Aug(\la)$. In this way, the augmentation variety can be thought of as a moduli space of fillings for a given Legendrian. The K\'alm\'an loop induces an automorphism of the augmentation variety and we can study this automorphism to understand the orbital structure of the K\'alm\'an loop action on fillings.

In this setting we introduce our second protagonist, a set of regular functions $\Delta_{i, j}$ on $\Aug(\la(A_{n-1}))$, which we show are indexed by diagonals of an $(n+2)$-gon. These regular functions admit an additional characterization as continuants, recursively defined polynomials studied by Euler in the context of continued fractions \cite{Euler}. This characterization leads to the appearance of a key supporting character, Euler's identity for continuants. Continuants naturally appear in the definition of the augmentation variety of $\la(A_{n-1})$ \cite{CGGS1}, and in Section \ref{section:algebraic} we show that the action of the K\'alm\'an loop is identical a special case of Euler's identity for continuants. In this sense, we may interpret the K\'alm\'an loop action on the augmentation variety as a Floer-theoretic manifestation of Euler's identity for continuants. Conversely, our geometric story may therefore be characterized as a somewhat convoluted proof of the continuant identity through contact geometry.


Paralleling the geometric story, we prove in Subsection \ref{sub:monomials} that the $\Delta_{i,j}$ give coordinate functions on the toric chart in $\Aug(\la(A_{n-1}))$ induced by a pinching sequence filling. From this algebraic argument, we conclude that the orbital structure of the K\'alm\'an loop corresponds precisely to the orbits of triangulations under rotation. The main results of this story are summarized in the two-part theorem below.

\begin{theorem}\label{thm: orbit size}
The action of the K\'alm\'an loop on the set $\mathcal{P}_n$, the Catalan number of exact Lagrangian fillings of $A$-type satisfies: 
\begin{enumerate}
    \item 	The number of K\'alm\'an loop orbits of fillings of $\la(A_{n-1})$ is 
		$$\frac{C_{n}}{n+2}+\frac{C_{n/2}}{2}+\frac{2C_{(n-1)/3}}{3} $$
		
		where the terms with $C_{n/2}$ and $C_{n/3}$ appear if and only if the indices are integers.\\
		
	\item   The $\Delta_{i, j}$ functions satisfy the equation $$\vartheta(\Delta_{1,k+1})+(-1)^{n-1}\Delta_{k,n+2}=-\Delta_{2,k}(\Delta_{1, n+2}-1)$$
	as polynomials in $\Z[z_1, \dots z_n]$.
\end{enumerate}
	\end{theorem}

Euler's identity for continuants plays a key role in the proof of (2). Following both of our proofs of Theorem \ref{thm: orbit size}, we conclude our exploration of the K\'alm\'an loop with a discussion of its combinatorial properties. In Subsection \ref{sub:orbits}, we describe a method for determining the orbit size of a filling based solely on the associated 312-avoiding permutation.

\begin{theorem}\label{thm: algorithm}
	There exists an algorithm of complexity $O(n^2)$ with input a 312-avoiding permutation $\sigma$ in $S_n$ for determining the orbit size of a pinching sequence filling $L_{\sigma}$ under the K\'alm\'an loop action.
	\end{theorem}

In addition, we give an entirely combinatorial description of the K\'alm\'an loop action in terms of 312-avoiding permutations as a sequence of edge flips of triangulations. The appearance of triangulated polygons and edge flips is perhaps best explained as the manifestation of the theory of cluster algebras lurking in the background. While cluster theory does not explicitly appear in any of our proofs, we highlight the connections where relevant.  We refer the reader to \cite{CGGLSS, CasalsWeng, GSW} for more dedicated treatments of cluster structures on the coordinate rings of algebraic varieties associated to Legendrian links.


 \begin{remark}
 As a possible application of Theorem \ref{thm: isotopy1} beyond Lagrangian fillings of $\la(A_{n-1})$, we might hope to describe the orbital structure of fillings of $\la(D_n)$ under the action of analogous Legendrian loops. In this context, the combinatorics of tagged triangulations are the $D$-type analog of the triangulations appearing in $A$-type. However, there is currently no known combinatorial bijection between tagged triangulations and the $D$-type Legendrian weaves constructed in \cite{Hughes2021}.  \hfill $\Box$
 \end{remark}

\subsection*{Organization} In Section \ref{sec: prelims}, we cover the necessary preliminaries, including the constructions of exact Lagrangian fillings, the Legendrian contact differential graded algebra, the K\'alm\'an loop, and related combinatorics. Section \ref{section:geometric} contains the proof of Theorem \ref{thm: isotopy1}, from which we obtain our geometric proof of Theorem \ref{thm: orbit size} as a corollary. In Section \ref{section:algebraic}, we give the Floer-theoretic proof of Theorem \ref{thm: orbit size}, featuring Euler's identity for continuants. Finally, Section \ref{section:combinatorics} presents the orbit size algorithm of Theorem \ref{thm: orbit size}, and we conclude with a combinatorial description of the K\'alm\'an loop action on these permutations.

	\subsection*{Acknowledgements} Many thanks to Roger Casals for his help and encouragement throughout. Thanks also to Lenny Ng for the original question about K\'alm\'an loop orbits that motivated this project. Thanks as well to the anonymous referees for their helpful feedback.

	\section{Preliminaries on Lagrangian fillings and their invariants}\label{sec: prelims}

	We begin with the necessary preliminaries from contact and symplectic topology. The standard contact structure $\xi_{st}$ in $\R^3$ is the 2-plane field given as the kernel of the 1-form $\alpha=dz-ydx$. A link $\lambda \subseteq (\R^3, \xi_{st})$ is Legendrian if $\lambda$ is always tangent to $\xi_{st}$. As $\la$ can be assumed to avoid a point, we can equivalently consider Legendrians $\la$ contained in the contact 3-sphere $(\mathbb{S}^3, \xi_{st})$ \cite[Section 3.2]{Geiges08}. 

	The symplectization of contact $\R^3$ is the symplectic manifold $(\R_t\times \R^3, d(e^t\alpha))$. Given two Legendrian links $\la_-$ and $\la_+$, an exact Lagrangian cobordism $L\subseteq (\R_t\times \R^3, d(e^t\alpha))$ from $\la_-$ to $\la_+$ is a cobordism $\Sigma$ such that there exists some $T>0$ satisfying the following: 

	\begin{enumerate}
		\item $d(e^t\alpha)|_\Sigma=0$ 
		\item $\Sigma\cap ((-\infty, T]\times \R^3)=(-\infty, T]\times \la_-$ 
		\item $\Sigma\cap ([T, \infty)\times \R^3)=[T, \infty) \times \la_+$ 
		\item $e^t\alpha|_\Sigma=df$ for some function $f$ on $\Sigma$. 
	\end{enumerate}

	An exact Lagrangian filling of the Legendrian link $\la\subseteq (\R^3, \xi_{st})$ is an exact Lagrangian cobordism $L$ from $\emptyset$ to $\la$ that is embedded in the symplectization $\R_t\times \R^3$. Equivalently, we consider $L$ to be embedded in the symplectic 4-ball with boundary $\partial L$ contained in the contact 3-sphere $(S^3, \xi_{st})$ \cite[Section 6.2]{ArnoldGivental01}. For $\la(A_{n-1})$, our fillings will be constructed as a series of saddle cobordisms and minimum cobordisms. 

We will depict a Legendrian link $\la$ in either of two projections. The front projection $\pi:(\R^3, \xi_{st}) \to \R^2$ given by $\pi(x, y, z)=(x, z)$ depicts a projection of $\la$ in the $xz-$plane. The Lagrangian projection $\Pi: (\R^3, \xi_{st})\to \R^2$ given by $\Pi(x, y, z)=(x,y)$ depicts $\la$ in the $xy-$plane. In the Lagrangian projection, crossings of $\Pi(\lambda)$ correspond precisely to Reeb chords of $\la$. Reeb chords are integral curves of the Reeb vector field $\partial_z$ that start and end on $\la$. In the front projection, the Legendrian condition $T_x\la \subseteq \ker(dz- y dx)$ implies that $y=\frac{dz}{dx}$. Therefore, Reeb chords are given by pairs of points $(x_1, z_1), (x_2, z_2)$ with $x_1=x_2$ and $\frac{dz_1}{dx_1}=\frac{dz_2}{dx_2}$. The key geometric content in Section \ref{section:geometric} will involve a careful comparison of Reeb chords in the front and Lagrangian projections of slicings of elementary cobordisms.

\subsection{Legendrian weaves}

In this subsection we describe Legendrian weaves, the first of two geometric constructions for producing exact Lagrangian fillings of a Legendrian link $\la(\beta)$. The key idea of Legendrian weaves is to combinatorially encode a \emph{Legendrian} surface $\Lambda$ in the 1-jet space $(J^1\D^2,\xi_{\st})=(T^*\D^2\times \R_z, \ker(dz-\alpha_{\st}))$, by the singularities of its front projection in $D^2\times \R_z$. The Lagrangian projection of $\Lambda$ then yields an exact Lagrangian filling. We describe the general case of Legendrian weave surfaces as given in \cite{CasalsZaslow} for the purposes of Theorem \ref{thm: isotopy1}. The case of 2-stranded positive braids is sufficient for all of the content pertaining to K\'alm\'an loop orbits in this paper.

The contact geometric setup of the Legendrian weave construction is as follows. 
Let $\beta$ be a positive braid and let $\Delta:=\prod_{i=1}^{N-1}\prod_{j=1}^{N-i} \sigma_j$ denote a half twist of the braid. We construct a filling of $\la(\beta)$ -- equivalently, the $(-1)$-framed closure of $\beta\Delta^2$, pictured in Figure \ref{fig: BraidClosure} (right) -- by first describing a local model for a Legendrian surface $\La$ in $J^1D^2$. By the Weinstein Neighborhood Theorem, a Legendrian embedding of $D^2$ into $(\R^5, \xi_{st})$ then gives rise to an embedding of $(J^1D^2, \xi_{st})$ into an open neighborhood of the image of $D^2$ under the embedding. In particular a Legendrian link in $J^1\partial D^2$ is mapped to a Legendrian link in the contact boundary $(\S^3,\xi_{\st})$ of symplectic $(\R^4,\omega_{\st}=d\theta_{\st})$ given as the co-domain of the Lagrangian projection $(\R^5,\xi_{\st})\lr(\R^4,\omega_st)$. Therefore, the boundary $\partial \La$ of our Legendrian surface will be taken to be a positive braid $\beta\Delta^2$ in $J^1S^1$. Under the contactomorphism described above, this positive braid is sent to the standard satellite of the standard Legendrian unknot. Diagramatically, this takes the braid $\beta\Delta^2$ in $J^1S^1$ to the $(-1)$-framed closure of $\beta\Delta^2$ in contact $S^3$.

\subsubsection{$N$-Graphs and Singularities of Fronts} To construct a Legendrian weave surface $\La$ in $J^1D^2,$ we combinatorially encode the singularities of its front projection in a colored graph. 
Local models for these singularities of fronts are classified by work of Arnold \cite[Section 3.2]{ArnoldSing}. The three singularities that appear in our construction describe elementary Legendrian cobordisms and are pictured in Figure \ref{fig: wavefronts}.

	\begin{center}
		\begin{figure}[h!]{ \includegraphics[width=.8\textwidth]{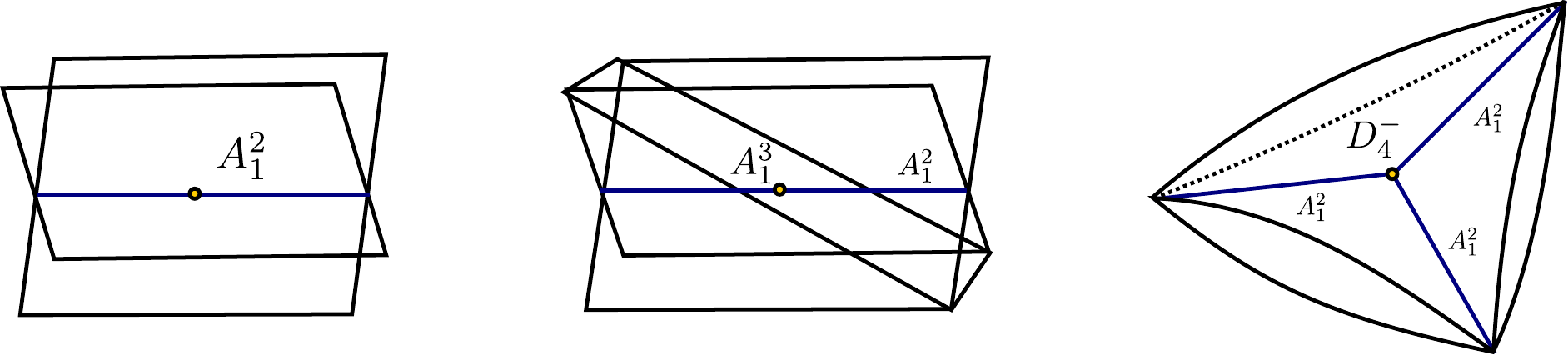}}\caption{Singularities of front projections of Legendrian surfaces. Labels correspond to notation used by Arnold in his classification.}
			\label{fig: wavefronts}\end{figure}
	\end{center}

Since the boundary of our singular surface $\pi(\La)$ is the front projection of an $N$-stranded positive braid, $\pi(\La)$ can be pictured as a collection of $N$ sheets away from its singularities. We describe the behavior at the singularities as follows:

\begin{enumerate}
    \item The $A_1^2$ singularity occurs when two sheets in the front projection intersect. This singularity can be thought of as the trace of a constant Legendrian isotopy in the neighborhood of a crossing in the front projection of the braid $\beta\Delta^2$. 
    \item The $A_1^3$ singularity occurs when a third sheet passes through an $A_1^2$ singularity. This singularity can be thought of as the trace of a Reidemeister III move in the front projection.
    \item A $D_4^-$ singularity occurs when three $A_1^2$ singularities meet at a single point. This singularity can be thought of as the trace of a 1-handle attachment in the front projection. 
\end{enumerate}

Having identified the singularities of fronts of a Legendrian weave surface, we encode them by a colored graph $\Gamma\subseteq D^2$. The edges of the graph are labeled by Artin generators of the braid and we require that any edges labeled $\sigma_i$ and $\sigma_{i+1}$ meet  at a hexavalent vertex with alternating labels while any edges labeled $\sigma_i$ meet at a trivalent vertex. To obtain a Legendrian weave $\Lambda(\Gamma)\subseteq (J^1\D^2,\xi_{\st})$ from an $N$-graph $\Gamma$, we glue together the local germs of singularities according to the edges of $\Gamma$. First, consider $N$ horizontal sheets $\D^2\times \{1\}\sqcup \D^2\times \{2\}\sqcup \dots \sqcup \D^2\times \{N\}\subseteq \D^2\times \R$ and an $N$-graph $\Gamma\subseteq \D^2\times \{0\}$. We construct the associated Legendrian weave $\Lambda(\Gamma)$ as follows  \cite[Section 2.3]{CasalsZaslow}.

	\begin{itemize}
		\item Above each edge labeled $\sigma_i$, insert an $A_1^2$ crossing between the $\D^2\times\{i\}$ and $\D^2\times \{i+1\}$ sheets so that the projection of the $A_1^2$ singular locus under $\pi:\D^2\times \R \to \D^2\times \{0\}$ agrees with the edge labeled $\sigma_i$.   
		\item At each trivalent vertex $v$ involving three edges labeled by $\sigma_i$, insert a $D_4^-$ singularity between the sheets $\D^2\times \{i\}$ and $\D^2\times\{i+1\}$ in such a way that the projection of the $D_4^-$ singular locus agrees with $v$ and the projection of the $A_2^1$ crossings agree with the edges incident to $v$.
		\item At each hexavalent vertex $v$ involving edges labeled by $\sigma_i$ and $\sigma_{i+1}$, insert an $A_1^3$ singularity along the three sheets in such a way that the origin of the $A_1^3$ singular locus agrees with $v$ and the $A_1^2$ crossings agree with the edges incident to $v$.
	\end{itemize} 
	\begin{center}
		\begin{figure}[h!]{ \includegraphics[width=\textwidth]{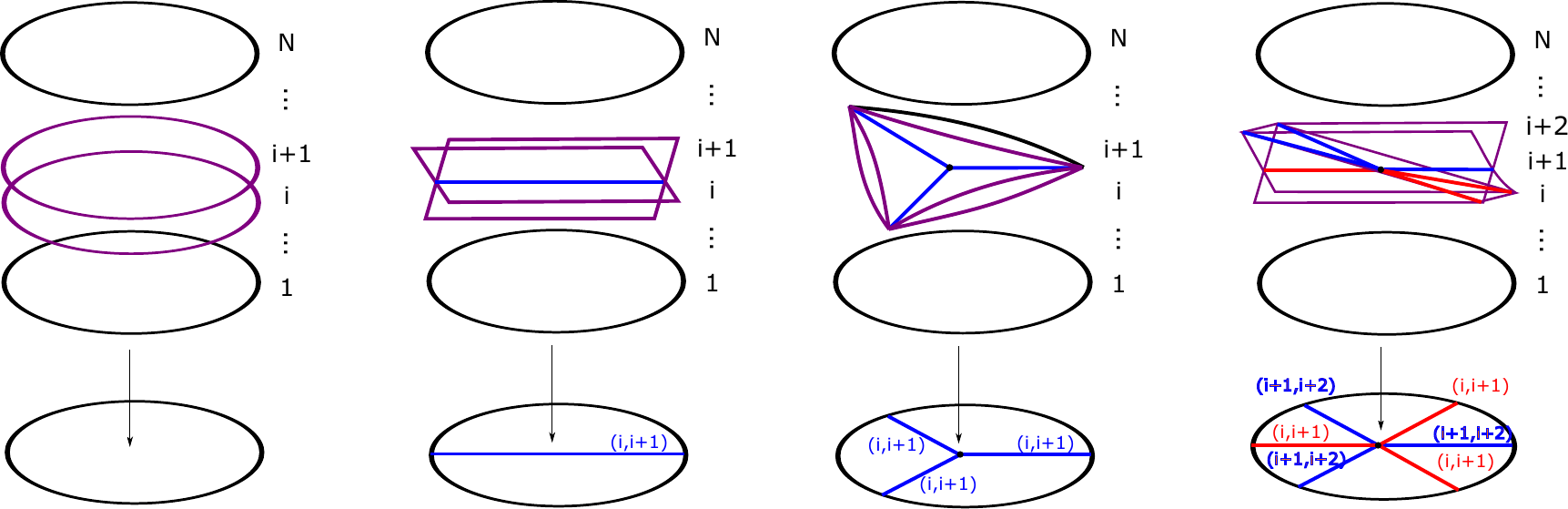}}\caption{The weaving of the singularities pictured in Figure \ref{fig: wavefronts} along the edges of the $N$-graph. Gluing these local pictures together according to the $N$-graph $\Gamma$ yields the weave $\Lambda(\Gamma)$.}
			\label{fig:Weaving}\end{figure}
	\end{center}
	
	If we take an open cover $\{U_i\}_{i=1}^m$ of $\D^2\times \{0\}$ by open disks, refined so that any disk contains at most one of these three features, we can glue together the resulting fronts according to the intersection of edges along the boundary of our disks. Specifically, if $U_i\cap U_j$ is nonempty, then we define $\pi(\La((U_1\cup U_2))$ to be the front resulting from considering the union of fronts $\pi(\La((U_1))\cup\pi(\La((U_j))$ in $(U_1\cup U_2)\times \R$.

	\begin{definition}\label{def: weave}
	    The Legendrian weave $\Lambda(\Gamma)$ is the Legendrian surface contained in $(J^1\D^2, \xi_{\st})$ with front $\pi(\La(\Gamma))=\pi(\La((\cup_{i=1}^m U_i))$ given by gluing the local fronts of singularities together according to the $N$-graph $\Gamma$.
	\end{definition}

	The immersion points of a  Lagrangian projection of a weave surface $\La$ correspond precisely to the Reeb chords of $\La$. In particular, if $\La$ has no Reeb chords, then $\Pi(\La)$ is an embedded exact Lagrangian filling of $\partial(\La)$. In the Legendrian weave construction, Reeb chords correspond to critical points of functions giving the difference of heights between sheets. Every weave surface in this paper admits an embedding where the distance between the sheets in the front projection grows monotonically in the direction of the boundary, ensuring that there are no Reeb chords. See also \cite[Subsection 2.3.2]{TreumannZaslow} 



\subsubsection{The $D_4^-$ cobordism}
Having described Legendrian weave surfaces, we now define a $D_4^-$ cobordism. As the $D_4^-$ singularity is not a generic Legendrian front singularity, we consider a generic perturbation of the $D_4^-$ singularity, as described in \cite[Remark 4.6]{CasalsZaslow} and pictured in Figure \ref{fig: D4Generic} (top). A slicing of the Legendrian front, depicted in Figure \ref{fig: D4Generic} (bottom), gives a movie of fronts describing the cobordism as follows.
Near a 
Reeb chord trapped between two crossings, we apply a Reidemeister I move and Legendrian isotopy to shrink the Reeb chord. We then add a 1-handle to remove this Reeb chord and apply another pair of Reidemeister I moves to simplify to a diagram with one fewer crossing than we started with. The trace of this movie of fronts forms a surface in $J^1[a, b]$ and yields an exact Lagrangian cobordism in symplectic $\R^4$ by taking the Lagrangian projection of its embedding in contact $\R^5$. By convention, we will identify the remaining crossing with the leftmost crossing of the original pair.

 	\begin{center}
		\begin{figure}[h!]{ \includegraphics[width=\textwidth]{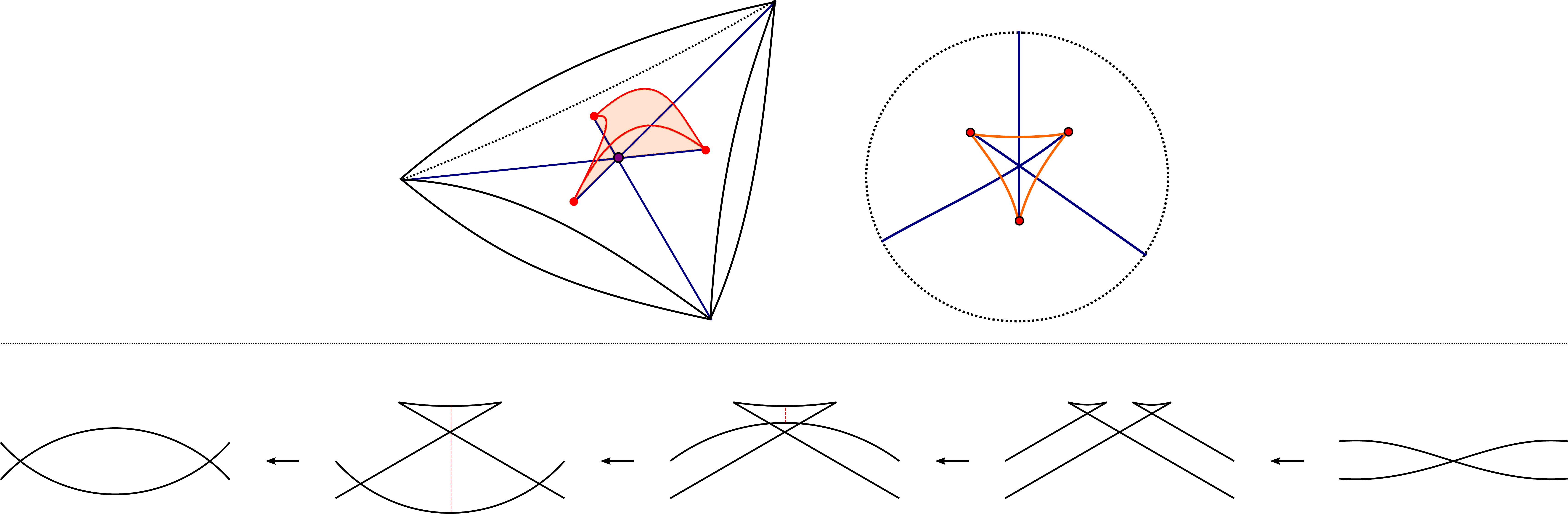}}\caption{A local model of a generic perturbation of the $D_4^-$ singularity as the front projection of a Legendrian surface (top) and as a movie of 1-dimensional fronts (bottom). The Reeb chord is depicted as a dashed red line. We first apply a Reidemeister I move before adding a 1-handle and applying two more Reidemeister I moves to arrive at a diagram with a single crossing.}
			\label{fig: D4Generic}\end{figure}
	\end{center}


\subsubsection{Vertical weaves}

In order to relate weave fillings to pinching sequence fillings, we will make use of an equivalent way of describing weaves, combinatorially presented in \cite{CGGS1}. This construction arranges the $N$-graph vertically, with $\partial D^2$ at the top and rest of the $N$-graph appearing below. This construction has the advantage of allowing us to unambiguously associate elementary Lagrangian cobordisms. 

Let $\Gamma\subseteq D^2$ be an $N$-graph and $\La(\Gamma)\subseteq J^1(D^2)$ be the associated weave. In order to produce the associated vertical weave $\La_V(\Gamma)$, we foliate the disk by copies of $S^1$, as shown in Figure \ref{fig: vertical} (left). We then consider a diffeomorphism $\varphi$ taking $D^2\backslash\{pt\}$ to $\S^1\times (-\infty, 0]$. We define $\varphi$ in such a way that the image of the foliation of $D^2\backslash\{pt\}$ is a foliation of $\R\times (-\infty, 0]$ by horizontal lines that are identified at $\pm \infty$ to form a foliation of $\S^1\times (-\infty, 0]$. The diffeomorphism $\varphi$ induces a contactomorphism $\Tilde{\varphi}:J^1(D^2)\to J^1(D^2)$. 

\begin{definition}
    The vertical weave $\La_V(\Gamma)$ is the Legendrian weave encoded (in the sense of Definition \ref{def: weave}) by the $N-$graph $\varphi(\Gamma)$. 
\end{definition}



After a planar isotopy of $\Gamma$, corresponding to a Legendrian isotopy of $\La(\Gamma)$, we can assume that there are no pairs of hexavalent or trivalent vertices appearing in the same horizontal strip $\R\times \{t\}$. The purpose of this modification is to unambiguously decompose a weave filling into elementary Lagrangian cobordisms in order to relate it to a decomposable Lagrangian filling in the symplectization of contact $\R^3$. Other than our manipulation of the ambient contact manifold, the vertical weave construction is identical to the Legendrian weaves described above. See Figure \ref{fig: vertical} for an example comparing Legendrian weave fillings of $\la(A_5)$.



\subsubsection{2-graphs and weave fillings of $\la(A_{n-1})$}
For our applications involving fillings of $\la(A_{n-1})$, we need only consider the case of $N=2$. 2-graphs are in combinatorial bijection with both binary trees and triangulations of $n+2$-gons, and we will make use of both of these bijections. In particular, we view the boundary of an $n+2$-gon as a disk so that the 2-graph dual to a triangulation represents a Legendrian weave surface with boundary $\la(A_{n-1})$. We list our choice of conventions below for ease of reference.

\begin{itemize}
    \item In a vertical weave, we encode $\la(\beta)$ with the braid word $\Delta\beta\Delta$ appearing at $\R\times \{0\}$.
    \item In a vertical weave, the edge exiting below a trivalent vertex with incoming edges $i$ and $i+1$ inherits the label $i$.
    \item In a 2-graph $\Gamma$ dual to a triangulation $\mathcal{T}$, the edge of $\Gamma$ most immediately clockwise from vertex $i$ of $\mathcal{T}$ is labeled by $i$. 
\end{itemize}

	\begin{center}
		\begin{figure}[h!]{ \includegraphics[width=\textwidth]{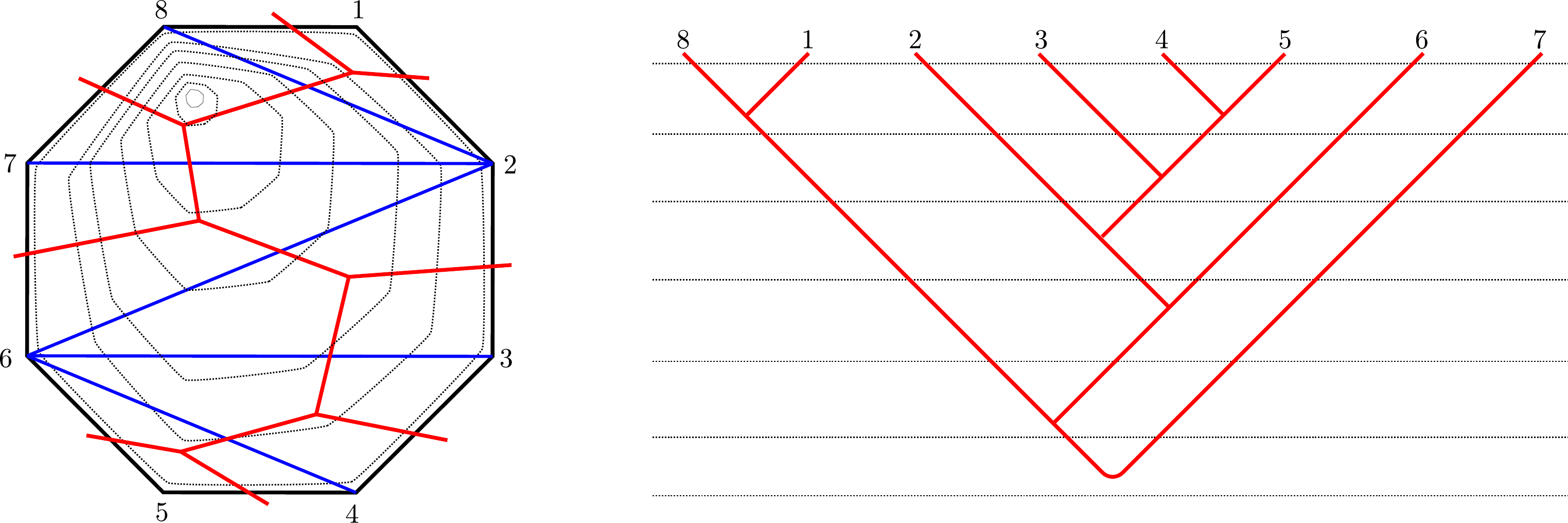}}\caption{A pair of 2-graphs representing the same weave filling of $\la(A_5)$. On the left, the 2-graph $\Gamma$ is inscribed in its dual triangulation of the octagon. On the right, the corresponding vertical weave is the image of the diffeomorphism $\varphi$. The edges of the vertical 2-graph are labeled by the nearest counterclockwise label of the dual triangulation. The dotted lines on the left give a foliation of $D^2$, corresponding to the foliation of $\R\times (-\infty, 0]$ depicted on the right.} 
			\label{fig: vertical}\end{figure}
	\end{center}

Note that our choice of labeling edges differs slightly from the conventions of \cite{CGGS1}. The choice of labeling given there corresponds to resolving the leftmost crossing of the pair in the $D_4^-$ cobordism. With our choice of conventions, we will be able to show that the combinatorial bijection described in Subsection \ref{Sub: combinatorics} relating 312-avoiding permutations to triangulations of the $n+2$-gon yields Hamiltonian isotopic fillings.

\subsection{The Legendrian contact DGA and the augmentation variety}

In this subsection, we describe the Legendrian contact DGA, a Floer-theoretic invariant of Legendrian knots and their exact Lagrangian fillings. We first give the Ekholm-Honda-K\'alm\'an construction for exact Lagrangian fillings cobordisms and then describe the necessary Floer-theoretic background.  

\subsubsection{The pinching cobordism and pinching sequence fillings}\label{sub: pinching cobordism}

The following definition gives a condition for being able to perform a pinching cobordism at a given crossing. See also \cite[Definition 6.2]{EHK} and \cite[Section 2.1]{CasalsNg}.

\begin{definition}
A crossing in the Lagrangian projection $\Pi(\la)$ of a Legendrian $\la$ is contractible if there is a Legendrian isotopy of $\la$ inducing a planar isotopy of $\Pi(\la)$ making the length of the corresponding Reeb chord arbitrarily small.  
\end{definition}


	We now describe the precise topological construction of the elementary cobordisms defining pinching sequence fillings. Consider a neighborhood of a contractible crossing depicted in the Lagrangian projection. 
Attaching a 1-handle at the crossing yields an exact Lagrangian cobordism in the symplectization $(\R_t\times \R^3, d(e^t (dz-ydx)))$ \cite[Section 6.5]{EHK}. In the Lagrangian projection, this 1-handle attachment is diagrammatically given as a 0-resolution of the crossing, as depicted in Figure \ref{fig: pinch}. If $\la$ is the rainbow closure of a positive braid, as is the case for $\la(A_{n-1})$, then every crossing of the braid is contractible \cite[Proposition 2.8]{CasalsNg}.  

	

		\begin{center}
		\begin{figure}[h!]{ \includegraphics[width=.4\textwidth]{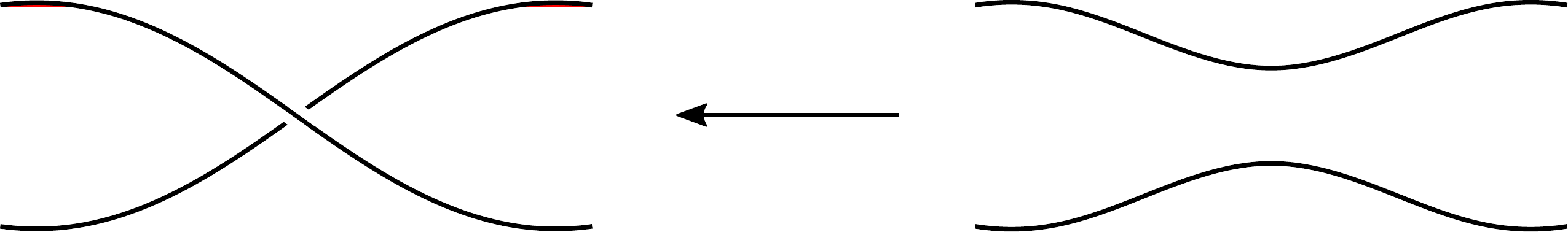}}\caption{A local model of a pinching cobordism as a 0-resolution of a contractible crossing in the Lagrangian projection. The direction of the arrow indicates a cobordism from the concave end to the convex end.}
			\label{fig: pinch}\end{figure}
	\end{center}
	
	Let us consider $\la\subseteq (\R^3, \xi_{st})$ and its front projection $\pi(\la)$. 
	In order to describe a pinching cobordism in terms of a projection of $\la$, we introduce the Ng resolution. This is a Legendrian isotopy $\la_t$ such that $\la_0=\la$ and the Lagrangian projection $\Pi(\la_1)$ can be obtained from the front projection $\pi(\la_1)$ by smoothing all left cusps and replacing all right cusps with small loops \cite{Ng03}. See Figure \ref{fig: T(2,n)} for an example. A pinching cobordism in the front projection of the link $\la(\beta)$ is then given by first taking the Ng resolution of $\la(\beta)$, performing a 0-resolution at a crossing in the Lagrangian projection as specified above, and then undoing the Ng resolution.

	 	\begin{center}
		\begin{figure}[h!]{ \includegraphics[width=.8\textwidth]{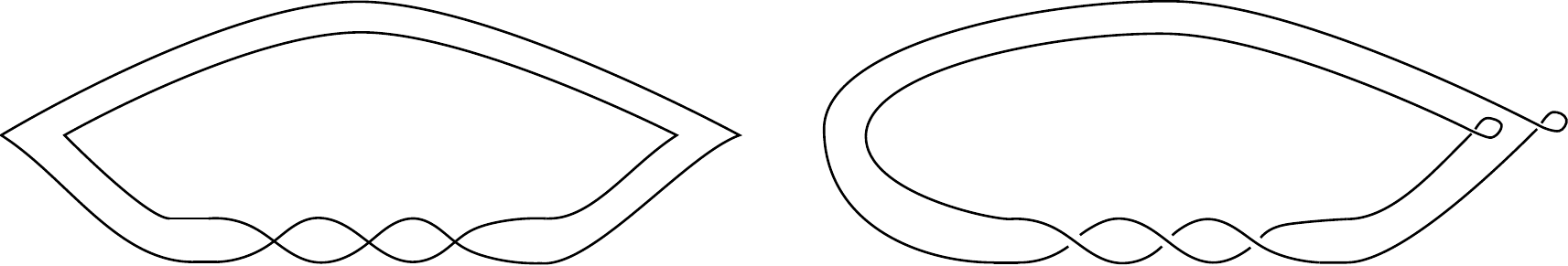}}\caption{A front projection of $\la(A_{2})$ (left) and its Ng resolution (right). The three leftmost crossings are contractible.}
			\label{fig: T(2,n)}\end{figure}
	\end{center}

Given that $\la(A_{n-1})$ has $n$ crossings, a pinching sequence filling can be characterized by a permutation $\sigma$ in $S_n$. Such a permutation specifies an order in which to apply these elementary cobordisms to the $n$ contractible 
crossings in the Ng resolution of $\la(A_{n-1})$. Given a permutation $\sigma$, we will denote it in one line notation $\sigma(1)\dots \sigma(n)$. If $\sigma$ is of the form $\sigma(1) \dots i \, j \dots k\dots \sigma(n)$ for $i>k>j$, the permutation $\sigma'=\sigma(1)\dots j\, i \dots k\dots\sigma(n)$ obtained by interchanging $i$ and $j$ gives an order of resolving crossings that yields the same Floer-theoretic invariant\footnote{The two fillings corresponding to $\sigma$ and $\sigma'$ yield identical augmentations $\epsilon_{\sigma}$ and $\epsilon_{\sigma'}$ of the DGA $\mathcal{A}(\la(A_{n-1}))$. This is because the presence of the crossing labeled $j$ prevents the existence of any holomorphic strip with positive punctures occurring at both crossings $i$ and $k$. Therefore, resolving crossing $k$ (resp. $i$) has no effect on the generator $z_i$ (resp. $z_k$) in the DGA $\mathcal{A}(\la(A_{n-1}))$.} \cite{YuPan}. This leads us to consider only a subset of permutations in $S_n$.

\begin{definition}
A 312-avoiding permutation is a permutation $\sigma\in S_n$ such that any triple of letters $i, j, k$ appearing in order in $\sigma$ does not satisfy the inequality $i>k>j$. 
\end{definition} 

Distinct 312-avoiding permutations yield distinct Hamiltonian isotopy invariants of exact Lagrangian fillings, i.e. restricting the indexing set from $S_n$ to 312-avoiding permutations yields the existence of at least a Catalan number of fillings of $\la(A_{n-1})$ up to Hamiltonian isotopy \cite[Theorem 1.1]{YuPan}.

	\subsubsection{The Legendrian contact DGA}\label{sub: DGA prelims}
	For a Legendrian link $\la$, the Legendrian contact differential graded algebra (DGA) is a powerful Floer-theoretic invariant of $\la$ \cite{Chekanov}. We denote the DGA of $\la$ by $\mathcal{A}(\la; R)$, or by $\mathcal{A}(\la)$ when we wish to suppress the dependence on the coefficient ring $R$. In our description below we will generally take $R$ to be $\Z[H_1(\la)]$ or $\Z[H_1(L)]$ for an exact Lagrangian cobordism $L$. 
	We refer the interested reader to \cite{EtnyreNg19} for a more general introduction to the Legendrian contact DGA and \cite[Section 3]{CasalsNg} for a discussion of different choices of coefficient rings. 
	
	To describe the algebra $\mathcal{A}(\la; R)$, we introduce the auxiliary data of a set of marked points on $\la$ and a corresponding set of capping paths.
	We label a marked point on the component $\la_a$ of $\la$ by $t_a^{\pm 1}$. Denote the collection of all such formal invertible variables by $T=\{t_1^{\pm 1}, \dots, t_m^{\pm 1}\}$. Given a Reeb chord $z$ with ends $z_a$, $z_b$ lying on components $\la_a$ and $\la_b$, a capping path $\gamma_z$ is the concatenation of paths following the orientation of $\la$ from $z_a$ to $t_a$ and $t_b$ to $z_b$. Here we require that $z_a$ corresponds to the undercrossing of $z$ in the Lagrangian projection. If we require one marked point for every component of $\la$, then we can think of $T$ as encoding $H_1(\la)$. The data of the DGA is then given as follows.  
	
	\textbf{Generators:} For a knot $\la$, the Legendrian contact DGA is freely generated over $R=\Z[T]$ by the Reeb chords of $\la\subseteq (\R^3, \xi_{st})$. In the Lagrangian projection, we can equivalently think of these generators as the crossings of $\la$.  

	\textbf{Grading:} Each $t_i$ and $t_i^{-1}$ is assigned grading 0. We define the grading for a Reeb chord $z$ as follows. As we traverse the capping path $\gamma_z$, the unit tangent vector to $\Pi(\la)$ makes a number of counterclockwise revolutions. We can perturb $\la$ in such a way that the tangent vectors at a crossing of $\Pi(\la)$ are always orthogonal and the number $r(\gamma_z)$ of such revolutions is always an odd multiple of $\frac{1}{4}.$ The grading of $z$ is then defined to be $|z|:=2r(\gamma_z)-\frac{1}{2}$. Grading is extended to products of generators $|yz|$ additively by $|yz|=|y|+|z|$. 

In the case of rainbow closures of a positive braid $\beta$, every Reeb chord that corresponds to a crossing of $\beta$ in the Ng resolution has degree zero while the remaining Reeb chords at the right of the diagram have degree one. 

	\textbf{Differential:} The differential is given by counts of certain holomorphic disks in the following way. We first decorate each quadrant of a crossing of $\Pi(\la)$ with two signs, a Reeb sign and an orientation sign. The Reeb sign is specified as pictured in Figure \ref{fig: ReebSigns} (left), where opposite quadrants have the same sign and adjacent quadrants have different signs. The orientation sign is given as in Figure \ref{fig: ReebSigns} (right), where the shaded regions are decorated with orientation sign $-$ and unshaded regions are decorated with orientation sign $+$.  
	
	The differential considers immersions $u$ from a punctured disk into $\R^2$ with boundary punctures on $\Pi(\la)$ up to reparametrization. We refer to any puncture appearing at a quadrant with a positive (resp. negative) Reeb sign as a positive (resp. negative) puncture. We restrict to immersions that have a single positive puncture and arbitrarily many negative punctures. For any such immersion $u$, denote by $w(u)$ the product of generators given by the negative boundary punctures/ 
	If the boundary of $u$ passes through any marked point $t_i$, then we obtain $w'(u)$ as the product of $w(u)$ by $t_i^{\pm 1}$. The power is assigned according to whether the orientation of $u$ at the relevant marked point agrees $(t_i)$ with the orientation of $\la$  or does not agree $(t_i^{-1})$ with the orientation of $\la$. To each disk $u$, we also assign the quantity $\sgn(u)\in \{\pm 1\}$ given by the product of the orientation signs appearing at boundary punctures of $u$. The differential at $z$ is then given by 
	$$\partial(z)=\sum \sgn(u)w'(u)$$ 
	where the sum is taken over all immersed disks $u$ with a single positive puncture at $z$. 
	We extend the differential to products $z_1z_2$ by the Leibniz rule $\partial (z_1z_2)=\partial(z_1)z_2+z_1\partial (z_2) $

		 	\begin{center}
		\begin{figure}[h!]{ \includegraphics[width=.6\textwidth]{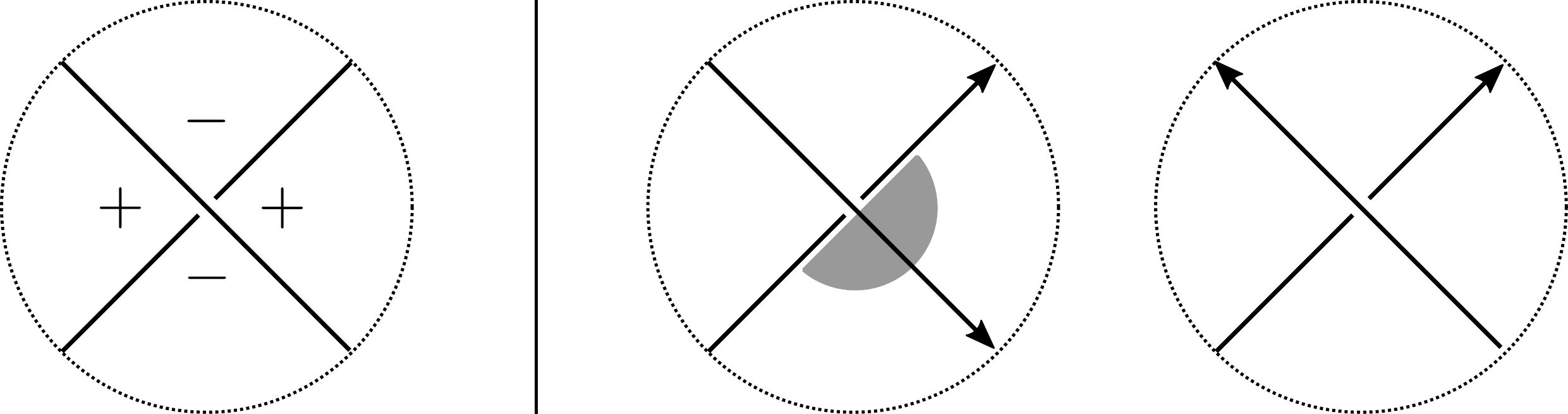}}\caption{Reeb signs (left) and orientation signs (right) at a crossing of $\Pi(\la)$. The quadrants shaded dark gray carry negative orientation signs, while the unshaded quadrants are positive.}
			\label{fig: ReebSigns}\end{figure}
	\end{center}


\begin{ex}
The DGA of $\la(A_2)$ is freely commutatively generated over $\Z[t, t^{-1}]$ by generators $a_1, a_2, z_1, z_2, z_3$, labeled in Figure \ref{fig: DGAEx}. The gradings are given by $|a_1|=|a_2|=1$ and $|z_i|=|t_1|=0$. The differential on generators $a_i$ is given by

$$\partial(a_i) = \begin{cases} z_1+z_3+z_1z_2z_3+t_1 & i=1 \\
1+t_1+z_2 + t_1z_3z_2+t_1z_1z_2+t_1z_1z_2z_3z_2 & i=2
\end{cases} $$
The differential on the remaining generators vanishes for degree reasons. Note that setting $t_1=-1$ implies $\partial(a_2)=-z_2\partial(a_1)$ \hfill $\Box$
\end{ex}

	\begin{center}
		\begin{figure}[h!]{ \includegraphics[width=.7\textwidth]{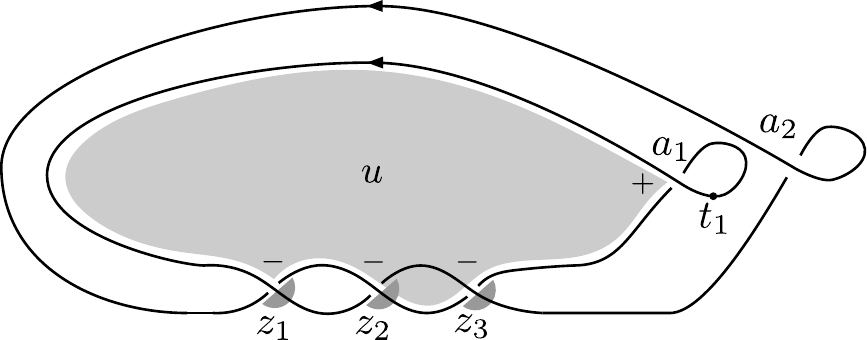}}\caption{The Lagrangian projection of the Legendrian trefoil decorated with Reeb signs and orientation signs. The light gray disk labeled $u$ has a positive puncture at $a_1$ and a single negative puncture at $z_1$. Thus, it corresponds to the term $z_1$ appearing in $\partial(a_1)$.}
			\label{fig: DGAEx}\end{figure}
	\end{center}

\subsubsection{Augmentations}\label{sub augs}

The Legendrian contact DGA can be difficult to extract information from, so it is often useful to consider augmentations of the DGA. Augmentations are DGA maps from $\mathcal{A}(\la)$ to some ground ring. Here we consider the ground ring of Laurent polynomials in $n-1$ variables with coefficients in $\Z$, understood as a DGA with trivial differential and concentrated in degree 0. In this subsection, we define augmentations and the related Legendrian isotopy invariant, the augmentation variety.

Augmentations of $\mathcal{A}(\la)$ are intimately tied to Lagrangian fillings of $\la$. This relationship can be understood through the functoriality of the DGA with respect to exact Lagrangian cobordisms. More precisely, Ekholm-Honda-K\'alm\'an show that an exact Lagrangian cobordism $L$ from $\la_-$ to $\la_+$ induces a DGA map $\Phi_L$ from $\mathcal{A}(\la_+; \Z_2)$ to  $\mathcal{A}(\la_-; \Z_2)$ \cite[Theorem 1.2]{EHK}. 
Their result was upgraded to make use of $\Z_2[H_1(L)]$ coefficients with an appropriate choice of marked points encoding $H_1(L)$ by Pan \cite[Proposition 2.6]{YuPan}. Pan's use of $H_1(L)$ coefficients is crucial for her ability to distinguish $C_n$ Lagrangian fillings of $\la(A_{n-1})$, as Ekholm-Honda-K\'alm\'an are only able to identify $(2^{n+1}-(-1)^{n+1})/3$ distinct Lagrangian fillings working over $\Z_2$ \cite[Theorem 1.6]{EHK}. The following result of Karlsson further improves Pan's coefficient ring to consider the augmentations over $\Z[H_1(L)]$.
	
	\begin{proposition}\label{prop: functoriality}\cite[Theorem 2.5]{Karlsson-cob} 
	An exact Lagrangian cobordism $L$ from $\la_-$ to $\la_+$ induces a DGA map $\Phi_L: \mathcal{A}(\la_+; \Z[H_1(\la_+)]\to \mathcal{A}(\la_-; \Z[H_1(L)])$.
	\end{proposition}

See also \cite[Section 3.3]{CasalsNg} for a discussion on Karlsson's choice of signs, as well as a geometric understanding of the induced map. As a result of Proposition \ref{prop: functoriality}, we can think of an augmentation of $\la$ as a map induced from $\mathcal{A}(\la; \Z[H_1(\la)])$ to the DGA of the empty set induced by an exact Lagrangian filling of $\la$. 

\begin{definition}
An augmentation $\epsilon_L$ induced by a Lagrangian filling $L$ of $\la$ is a DGA map 
	$$\epsilon_L: \mathcal{A}(\la, \Z[H_1(\la)]) \to \Z[H_1(L)] $$ where we think of $\Z[H_1(L)]$ as a DGA concentrated in degree zero with trivial differential.
	\end{definition}




	The functoriality of the DGA motivates the study of augmentations of $\mathcal{A}(\la)$ in order to better understand Lagrangian fillings of $\la$. The space of all augmentations of $\mathcal{A}(\la)$, denoted by $\Aug(\la)$, is an invariant of $\la$. In the case where $\la$ is the rainbow closure of a positive braid, $\Aug(\la)$ has the structure of an affine algebraic variety and is known as the augmentation variety. We tensor our coefficients ring with $\C$ in order to consider augmentations over a field.\footnote{To clarify, complexifying is solely for the purpose of simplifying the algebro-geometric discussion in this paragraph. For all other purposes, we will continue to use integer coefficients.} When the grading of $\mathcal{A}(\la)$ is concentrated in non-negative degrees, as is the case for rainbow closures of positive braids, then $\Aug(\la)\cong  \Spec H_0(\mathcal{A}(\la))$, see e.g. \cite[Corollary 2.9]{GSW}. Since $\Spec$ is contravariant, $\epsilon_{L}$ induces a map $\Spec(\C[s_1^{\pm 1}, \dots, s_{b_1(L)}^{\pm 1}])\to \Spec H_0(\mathcal{A}(\la; \C[H_1(\la)]))$, where we have identified the ground ring of Laurent polynomials with complex coefficients $\C[H_1(L)]$ with the group ring 
	$\C[s_1^{\pm 1}, \dots, s_{b_1(L)}^{\pm 1}]$. We interpret this map as the inclusion of a toric chart $(\C^\times)^{b_1(L)}$ into the augmentation variety 
	$$\Spec(\C[s_1^{\pm 1}, \dots, s_{b_1(L)}^{\pm 1}])\cong (\C^\times)^{b_1(L)}\hookrightarrow \Aug(\la).$$  

	The image of degree-zero generators under an augmentation give local coordinate functions on the corresponding toric chart.
	In order to describe these local coordinate functions, we discuss Pan's explicit computation of induced DGA maps in the context of Lagrangian fillings of $\la(A_{n-1})$ with a lift to $\Z[H_1(L)]$ following \cite[Section 4.2]{CasalsNg}. For a pinching cobordism, the induced map is given by a certain count of holomorphic disks, similar to the differential. The homology coefficients are determined by the intersection of these disks with relative homology classes in $H_1(L, \la_- \sqcup \la_+)$, which is identified with $H_1(L)$ via Poincar\'e duality.

	 Pan describes a set of generators for $H_1(L, \la_- \sqcup \la_+)$ for a sequence of pinching cobordisms. In this setting, a relative homology cycle $\gamma_{\sigma(i)}$ starts from the saddle point originally labeled $z_{\sigma(i)}$ and extends downwards to $\la_-$ where it meets the boundary in $s_{\sigma(i)}$ and $s_{\sigma(i)}^{-1}$. In order to consider signs, we orient this relative cycle so that the two halves of the cycle are labeled by $s_{\sigma(i)}$ and $-s_{\sigma(i)}^{-1}$, as in Figure \ref{fig: HomologyCoeffs}. In a slicing of the symplectization, $\gamma_{\sigma(i)}$ meets the Lagrangian projection of $\la(A_{k})$ in two points labeled $s_{\sigma(i)}$ and $-s_{\sigma(i)}^{-1}$, so that in practice, these generators reduce the computation of the coefficients to a combinatorial count of marked points. 
	 
	 \begin{center}
		\begin{figure}[h!]{ \includegraphics[width=.3\textwidth]{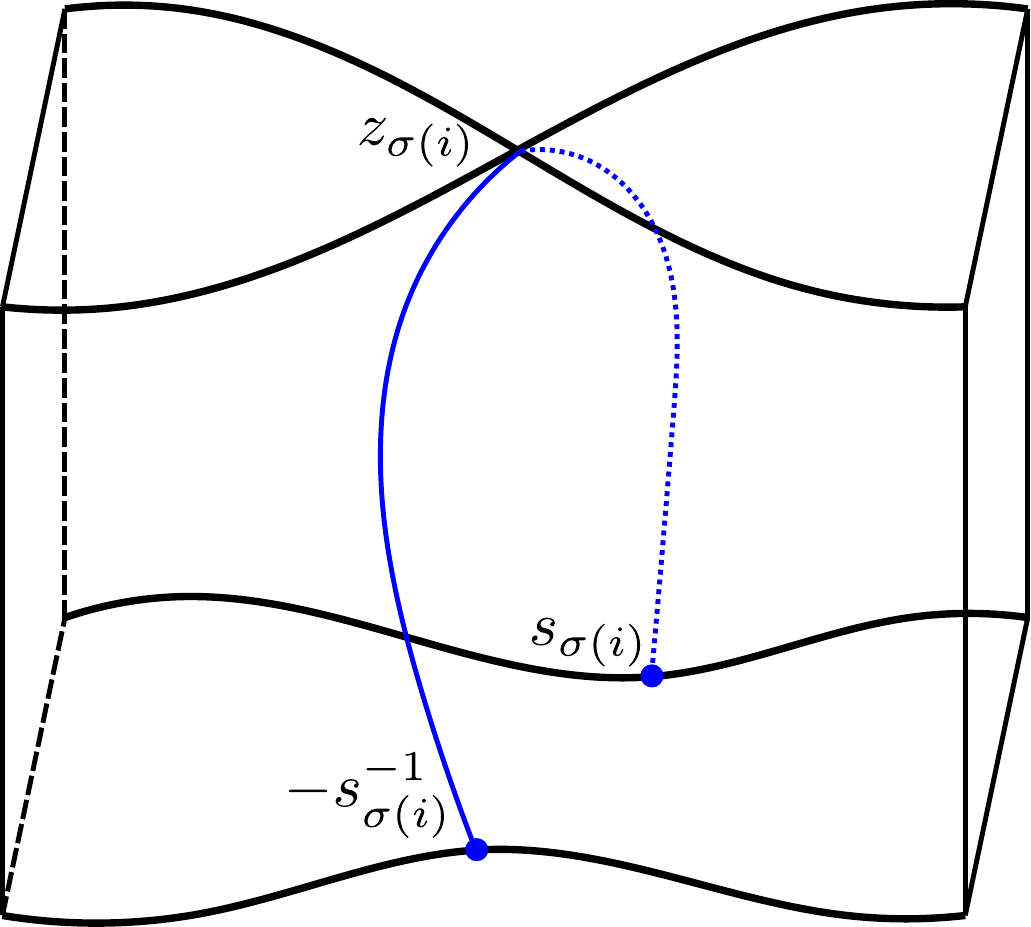}}\caption{A local model of a relative cycle encoding the homology of a pinching cobordism $L$. At the top of the figure, the length of the Reeb chord $z_{\sigma(i)}$ is 0, i.e. the two strands intersect at the point $z_{\sigma(i)}$. The bottom of the figure depicts a 0-resolution of the crossing where marked points labeled $-s_{\sigma(i)}^{-1}$ and $s_{\sigma(i)}$ encode $H_1(L)$.}
			\label{fig: HomologyCoeffs}\end{figure}
	\end{center}

	 Given a pinching cobordism $L_{\sigma(i)}$ at the Reeb chord $z_{\sigma(i)}$ as part of a Lagrangian filling $L_\sigma$, the induced map $\Phi_{i}$ on the generator $z_j$ is computed as a sum over all immersed disks with positive punctures at both $z_{\sigma(i)}$ and $z_j$. As before, we denote by $w'(u)$ the product of negative punctures of the immersed disk $u$ and intersections of $u$ with marked points counted with orientation.
	 
	 \begin{definition}
	  The DGA map $\Phi_i$ induced by a pinching cobordism at the Reeb chord $z_{\sigma(i)}$ is given by 	  $$\Phi_i(z_j)=z_j+\sum \sgn(u)w'(u)$$ 
	  	 The Reeb chord $z_{\sigma(i)}$ is sent to $s_{\sigma(i)}$ by $\Phi_{i}$.
	 \end{definition}


	 	 \begin{center}
		\begin{figure}[h!]{ \includegraphics[width=.7\textwidth]{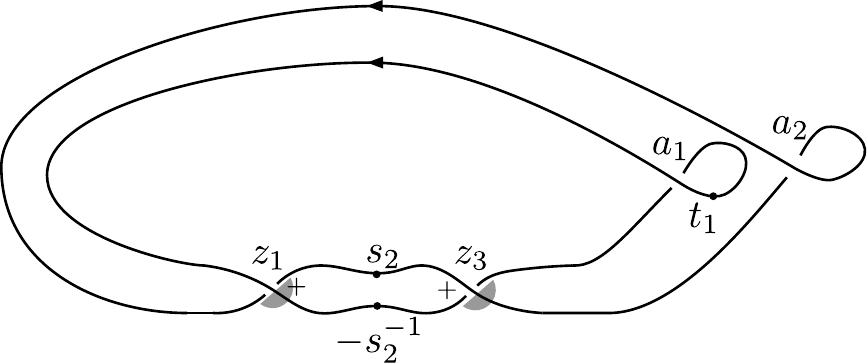}}\caption{A Legendrian Hopf link obtained from the knot pictured in Figure \ref{fig: DGAEx} by pinching at the Reeb chord labeled by $z_2$.  }
			\label{fig: PinchingEx}\end{figure}
	\end{center}
	 
\begin{ex}
Pinching at the Reeb chord labeled by $z_2$ of the Legendrian trefoil pictured in Figure \ref{fig: DGAEx} yields the Legendrian Hopf link, pictured in Figure \ref{fig: PinchingEx} with the addition of marked points $s_2$ and $-s_2^{-1}$. The induced map $\Phi_1$ on the DGA is given by $$z_1\mapsto z_1-s_2^{-1}, \qquad z_2\mapsto s_2,\qquad z_3\mapsto z_3-s_2^{-1}.$$ Performing another pinch at $z_1$ induces the map $\Phi_2$ given by $$z_1\mapsto s_1, \qquad s_2\mapsto s_2, \qquad z_3\mapsto z_3+s_1^{-1}s_2^{-2}.$$ 

The map on $z_3$ is determined by the disk with positive punctures at $z_1$ and $z_3$ passing through $s_2$ and $-s_2^{-1}$. \hfill $\Box$
\end{ex}

    Pan gives a purely combinatorial description of the map $\Phi_i$ induced by opening the crossing labeled $z_{\sigma(i)}$ \cite[Definition 3.2]{YuPan}. First, define the set
		
		\begin{align*}
			T^i_\sigma&=\{j\in \{1, \dots, n\}| \sigma^{-1}(j)>\sigma^{-1}(i) \text{ and if } i<k<j \text{ or } j<k<i, \text{ then } \sigma^{-1}(k)<\sigma^{-1}(i)\}.
		\end{align*}
		
		For $j\in T^{\sigma(i)}_\sigma$ and $1\leq i\leq n$, the DGA map is given by
		\begin{equation*}
			\Phi_i(z_j)=z_j+ s_{\sigma(i)}^{-1}\prod_{\substack{j<k<\sigma(i) \text{  or }\\ \sigma(i)<k<j}} s_k^{-2}
		\end{equation*} and for $j=\sigma(i),$ $\Phi_i(z_j)=s_j.$ Otherwise, we take $\Phi_i$ to be the identity. 
		\begin{lemma}
		The lift of Pan's combinatorial formula for $\Phi_i$ to $\Z[H_1(L)]$ 
		is given by
		\begin{equation*}\label{eq: Phi}
			\Phi_i(z_j)=z_j+ (-1)^{|j-\sigma(i)|+1}s_{\sigma(i)}^{-1}\prod_{\substack{j<k<\sigma(i) \text{  or }\\ \sigma(i)<k<j}} s_k^{-2}.
		\end{equation*} 
		
		\end{lemma}
		\begin{proof}
		    To upgrade Pan's formula to $\Z[H_1(L)]$ coefficients, we  note that the number of pairs of marked points that appear between $z_j$ and $z_{\sigma(i)}$ is precisely $|j-\sigma(i)|$. Since each pair of marked points contributes a $-1$ factor to $\sgn(u)$ and the disk $u$ with positive punctures at $z_j$ and $z_{\sigma(i)}$ picks up an additional $-1$ factor from the orientation sign of the leftmost positive puncture, we arrive at the formula given above.
		\end{proof}

	To compute an augmentation of $\mathcal{A}(\la; \Z[T])$, we also need to describe the map induced by the minimum cobordism. This minimum cobordism is given by filling a standard Legendrian unknot $\la_U$ with an exact Lagrangian disk. It induces a map $\Phi_{min}$ sending the Reeb-chord generator of $\mathcal{A}(\la_U; \Z[T])$ to 0. The map on the marked point generators can be deduced from the fact that $\Phi_{min}$ is a DGA map and therefore we must have $\Phi_{min}\circ\partial=0$. In the context of a filling $L_\sigma$ of $\la(A_{n-1})$, this tells us that $s_1\dots s_n+t_1=0$. For knots, we also obtain $-(s_1\dots s_n)^{-1}+1=0$, implying that $t_1$ is mapped to $-1$. For links, we have $(s_1\dots s_n)^{-1}+t_2=0$, implying only that $t_1t_2=1$. Pan avoids this ambiguity by computing augmentations of $\mathcal{A}(\la(A_{2n-1};H_1(\la(A_{2n-1}))$ induced by  fillings $L_\sigma$ of $\la(A_{n})$ where $\sigma(1)=n+1$. This is equivalent to setting $t_1=-1$, from which we obtain $t_2=-1$. Therefore, the DGA map induced by the minimum cobordism is given by the following formula.
$$\Phi_{min}(s_n)=s_1^{-1}\dots s_{n-1}^{-1}$$ 
For $s_i\neq s_n$, $\Phi_{min}$ is the identity map. For marked points $t_i,$ we have $\Phi_{min}(t_i)=-1$.

	
Together, the maps $\Phi_i$ and $\Phi_{min}$ give us the ingredients to define the augmentation induced by a pinching sequence filling of $\la(A_{n-1}).$

\begin{definition}
  The augmentation $\epsilon_{L_\sigma}$ induced by the Lagrangian filling $L_\sigma$ of $\la(A_{n-1})$ is given by the DGA map
  $$\epsilon_{L_\sigma}:=\Phi_{min}\circ\Phi_n\circ \dots \circ \Phi_1$$ 
\end{definition}


			
			
			
To simplify our computations involving the DGA and the K\'alm\'an loop, we will always set $t_1=-1$ and $t_2=-1$ for the remainder of this manuscript. By our definition of $\epsilon_{L_\sigma}$, this does not affect the augmentation induced by $L_\sigma$.   
	
\subsubsection{Braid matrices}\label{sub: braid matrices}

	For $\la(A_{n-1})$, the polynomials defining the augmentation variety have a combinatorial description as a specific entry in a product of matrices. These matrices originally appeared in \cite{Kalman-braid} as a means for encoding the immersed disks contributing to the differential. More recently, they were used in \cite{CGGS1} to give a holomorphic symplectic structure on the augmentation variety. We adopt the conventions of \cite{CGGS1} in defining the braid matrix.
	
	\begin{definition}
	    The braid matrix $B(z_i)$ is given by
	    $B(z_i):=\begin{pmatrix}0 & 1 \\ 1 & z_i\end{pmatrix}.$ 
	\end{definition}

Intuitively, one can think of the matrix $B(z_i)$ as encoding whether or not an immersed disk has a negative puncture at the crossing labeled by $z_i$. A product of braid matrices can be used to compute the differential of $\mathcal{A}(\la(A_{n-1}; \Z)$ as follows. 
First, label the crossings of $\la(A_{n-1})$ by $a_1, a_2, z_1, \dots z_n$, as in Figure \ref{fig: DGAEx}. From \cite[Section 3.1]{Kalman-braid}, we have that $\partial(a_1)=-1+\left[\prod_{i=1}^n B(z_i)\right]_{(2,2)}$ 
where the subscript denotes the $(2, 2)$ entry of the product and the $-1$ appears due to our choice of conventions. See also \cite[Proposition 5.2]{CasalsNg} for a similar computation. An analogous computation to the case of $\la(A_2)$ implies that the differential of $a_2$ is given by $\partial(a_2)=-\Delta_{2, n+1}(\partial(a_1))$. 

The computation of the differential via braid matrices also allows us to express the augmentation variety $\Aug(\la(A_{n-1}))$ in a similar manner. As augmentations are DGA maps, they satisfy $\epsilon\circ \partial=\partial \circ \epsilon$. Since $\epsilon$ respects the grading, it must vanish on generators of nonzero degree, implying that for such generators $a$, we have $\partial\circ \epsilon(a) = 0$. Therefore, any augmentation $\epsilon$ of $\mathcal{A}(\la(A_{n-1}))$ satisfies $\epsilon\circ \partial(a_1)=\partial(a_2)=0$. Since $\partial(a_2)$ is a multiple of $\partial(a_1)$, the vanishing of $\partial(a_1)$ is both necessary and sufficient to satisfy the vanishing condition for $\epsilon\circ \partial$, so that the augmentation variety is cut out by the vanishing of the equation $\partial(a_1)=0$.  

\begin{lemma}
	The augmentation variety $\Aug(\la(A_{n-1}))$ is the zero set of the polynomial 
		$$X_n:=-1+\left[\prod_{i=1}^n B(z_i)\right]_{(2,2)}$$
\end{lemma}

	In addition to computing the augmentation variety, braid matrices also define regular functions $\Delta_{i,j}$ on $\Aug(\la(A_{n-1}))$ that will play an important role in Section \ref{section:algebraic}.
	
	\begin{definition}
	    The regular function $\Delta_{i, j}\in \Z[X_n]$ is given by 
	    $\Delta_{i,j}:=\left[\prod_{k=i}^{j-2}B(z_k)\right]_{(2,2)}.$
	\end{definition}

	We specify the value of $\Delta_{i, i+1}$ to be 1. We collect some useful identities relation the $\Delta_{i, j}$ functions to the theory of continuants below. 

\begin{ex}
Consider the Legendrian trefoil, $\la(A_2)$. The augmentation variety $\Aug(\la(A_2))$ is the zero set of the polynomial $X_3=-1+z_1+z_3+z_1z_2z_3$. The regular functions $\Delta_{i, j}$ are of the form $\Delta_{i, i+2}=z_i$ or $\Delta_{i, i+3}=1+z_iz_{i+1}$, for $1\leq i\leq 3$. 
\hfill $\Box$
\end{ex}

\subsubsection{Continuants}	

Continuants are a family of polynomials $K_n(x_1, \dots x_n)$ studied by Euler in his work on continued fractions \cite{Euler}. Continuants are defined by the following recursive formula:

$$K_n(x_1, \dots x_n)=x_1K_{n-1}(x_2, \dots x_{n})+ K_{n-2}(x_3, \dots x_{n})$$ 

$K_0()=1, K_1(x_1)=x_1.$

As mentioned above, the regular functions $\Delta_{i, j}$ are related to continuants.

\begin{lemma}
Let $n=j-2-i$ and $x_k=z_{i+k-1}$. Then $$K_n(x_1, \dots x_n)=\Delta_{i,j}.$$

\end{lemma}
\begin{proof}

The following is a classical property of continuants (see e.g. \cite[Section 1]{Frame49}) that allows us to understand the defining recursion relation in terms of braid matrices. 
 $$\begin{pmatrix}K_{n-2}(x_2, \dots, x_{n-1})& K_{n-1}(x_2, \dots, x_n)\\
K_{n-1}(x_1, \dots x_{n-1}) & K_n(x_1, \dots, x_n) \end{pmatrix}=B(x_1)\dots B(x_n).$$ 
This follows inductively from applying the recursion relation in computing the matrix product
$$\begin{pmatrix} 0 & 1 \\ 1 & x_1\end{pmatrix}\begin{pmatrix}K_{n-3}(x_3, \dots, x_{n-1})& K_{n-2}(x_3, \dots, x_n)\\
K_{n-2}(x_2, \dots x_{n-1}) & K_{n-1}(x_2, \dots, x_n) \end{pmatrix}=\begin{pmatrix}K_{n-2}(x_2, \dots, x_{n-1})& K_{n-1}(x_2, \dots, x_n)\\
K_{n-1}(x_1, \dots x_{n-1}) & K_n(x_1, \dots, x_n) \end{pmatrix}.$$

Therefore, $$K_{n}(x_1, \dots, x_{n})=\left[\prod_{k=1}^{n}B(x_k)\right]_{(2, 2)}$$
Replacing $x_k$ with $z_{i+k-1}$ yields the desired identification.
\end{proof}

As a consequence, we obtain the continuant recursion relation in the context of the $\Delta_{i, j}$ functions.
 \begin{equation}\label{eq:recurse}
\Delta_{i, j}=z_i\Delta_{i+1,j}+\Delta_{i+2, j}
\end{equation}

Continuants satisfy several identities, the most general of which is Euler's identity for continuants. We present this identity in the context of the $\Delta_{i, j}$ functions: 
	$$\Delta_{1, \mu+\nu+2}\Delta_{\mu+1, \mu+\kappa+2}-\Delta_{1, \mu+\kappa+2}\Delta_{\mu+1, \mu+\nu+2}=(-1)^{\nu+1}\Delta_{1, \mu+1}\Delta_{\mu+\kappa+2, \mu+\nu+2}$$ for
	$\mu\geq 1, \kappa\geq 0, \nu\geq \kappa+1$ \cite{Ustinov06}.
We require a special case of this identity for our algebraic proof of Theorem \ref{thm: orbit size}. Namely, when $\mu=1, \kappa=k-3\geq 0, \nu=n-1\geq k-2$, we obtain
\begin{equation}\label{eq Euler}
\Delta_{1, n+2}\Delta_{2, k}-\Delta_{1, k}\Delta_{2, n+2}=(-1)^{n}\Delta_{1, 2}\Delta_{k+2, n+2}.
\end{equation}


	\subsubsection{The K\'alm\'an loop}\label{sub: Kalman prelims}
	
	In \cite{Kalman}, K\'alm\'an defined a geometric operation on Legendrian torus links that induces an action on their exact Lagrangian fillings. In the case of $\la(A_{n-1})$, this operation consists of a Legendrian isotopy that is visualized by dragging the leftmost crossing clockwise around the link until it becomes the rightmost crossing. The graph of this isotopy is an exact Lagrangian cylinder in the symplectization of $(\R^3, \xi_{st})$. 
	Concatenating this cylinder with a Lagrangian filling $L$ of $\la(A_{n-1})$ yields another filling, generally not Hamiltonian isotopic to $L$.	
	As computed in \cite[Proposition 9.1]{Kalman}, this induces a map on the DGA $\mathcal{A}(\la(A_{n-1}); \Z_2)$, which in turn induces an automorphism $\vartheta$ on the augmentation variety $\Aug(\la(A_{n-1}))$. Following \cite[Section 3]{CasalsNg}, we can compute this induced action with integer coefficients. 
	The additional information of this integral lift consists solely of a choice of signs for terms in the image of $\vartheta$, as can be seen in K\'alm\'an's example computation over $\Z[t, t^{-1}]$ in the case of the $\la(A_2)$ \cite[Section 5]{Kalman}.\footnote{Note that K\'alm\'an uses a different choice of sign conventions than Casals and Ng. By \cite[Proposition 3.14]{CasalsNg}, these different sign conventions yield equivalent induced augmentations.} The map $\vartheta$ on generators $z_i$ is then given by 
	$$\vartheta(z_i)= \begin{cases} -\Delta_{2, n+2} & i=1\\ z_{i-1} & 2 \leq i \leq n
	\end{cases}$$


 In the $\Delta_{i, j}$ functions, this is expressed as $\vartheta(\Delta_{i, j})=\Delta_{i-1, j-1}$ for $i>1$, and $$\vartheta(\Delta_{1, j})=-\left[B(\Delta_{2, n+2})\prod_{i=1}^{j-3}B(z_i)\right]_{(2,2)}.$$

	\subsection{Combinatorics of triangulations}\label{Sub: combinatorics}
	
	The two geometric constructions of Lagrangian fillings discussed above, as well as their algebraic invariants involve combinatorial characterizations that are crucial for our later description of the action of the K\'alm\'an loop. This subsection starts with a description of the specific combinatorial bijection between 312-avoiding permutations $\sigma$ and triangulations $\mathcal{T}_\sigma$ that we use to relate the Lagrangian fillings $L_\sigma$ and $L_{\mathcal{T}_\sigma}$. We then describe the orbital structure of triangulations under the action of rotation, and conclude with a combinatorial description of the set of triangulations in terms of the flip graph.

		\subsubsection{The clip sequence bijection}
	As a first step towards relating pinching sequence fillings and weave fillings, we describe a combinatorial bijection between triangulations of the $n+2$-gon and 312-avoiding permutations in $S_n$. As a corollary to Theorem \ref{thm: isotopy1}, we will show that this combinatorial bijection corresponds to a Hamiltonian isotopy of the Lagrangian fillings defined by this input data.
 
	For a 312-avoiding permutation $\sigma$, we denote the corresponding triangulation by $\mathcal{T}_\sigma$ and a diagonal between vertex $i$ and vertex $j$ of $\mathcal{T}_\sigma$ by $D_{i,j}$. Adopting the terminology of \cite{RegevAlon2013Abbt}, we refer to a triangle  in $\mathcal{T}_\sigma$ with sides $D_{i, i+2}, D_{i, i+1}, D_{i+1, i+2}$, two of which lie on the $(n+2)$-gon, as an ear of the triangulation. Note that any triangulation must have at least two ears and that the middle vertex of an ear necessarily has no diagonal incident to it. 
	
	Given a triangulation of the $(n+2)$-gon, the clip sequence bijection is defined as follows. First, label the vertices in clockwise order from 1 to $n+2$. Remove the middle vertex of the ear with the smallest label, record the label and delete all edges of the $(n+2)$-gon incident to the vertex. Repeat this process with the ear whose middle vertex is now the smallest of the remaining vertices in the resulting triangulation of the $n+1$-gon. Continue this process until no triangles remain. The main result of \cite{RegevAlon2013Abbt} is that this map is indeed a bijection between the set of 312-avoiding permutations in $S_n$ and triangulations of the $(n+2)$-gon. 
	
	The clip sequence  bijection allows us to explicitly define a weave filling with the input of a 312-avoiding permutation $\sigma$.
	
	\begin{definition}
		The Lagrangian filling $L_{\mathcal{T}_\sigma}$ is the weave filling defined by the 2-graph dual to the triangulation $\mathcal{T}_\sigma$.
	\end{definition}

\vspace{-.2cm}
	
	See Figure \ref{fig: clip} for a computation of the 312-avoiding permutation corresponding to the triangulation dual to the 2-graph example given above. 
	
			\begin{center}
		\begin{figure}[h!]{ \includegraphics[width=.6\textwidth]{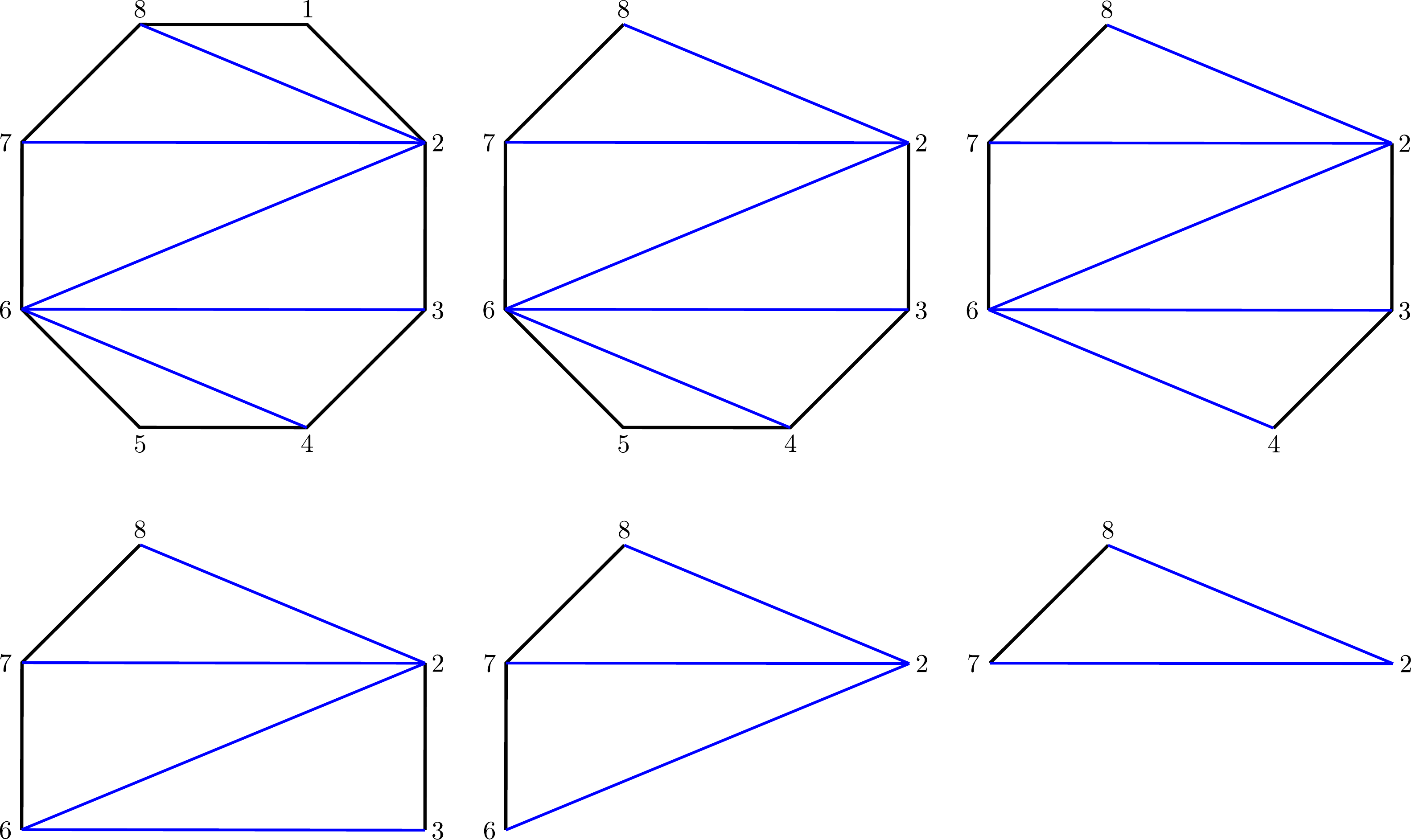}}\caption{An example computation of the clip sequence bijection. Starting with our initial triangulation, we remove and record the smallest numbered vertex with no incident diagonals. From the sequence pictured, we get the 312-avoiding permutation $\sigma=1\,5\,4\,3\,6\,2$. 
		The diagonal $D_{2, 8}$ yields the function $\Delta_{1, 3}$ after adding 1 to both indices and reducing mod 8.}
			\label{fig: clip}\end{figure}
	\end{center}

	
	

\subsubsection{Orbits of triangulations under rotation}\label{sub: Orbital structure}

In Section \ref{section:algebraic}, we will show that in $\Aug(\la)$, the global functions $\Delta_{i,j}$ transform as $\vartheta(\Delta_{i,j})=\Delta_{i-1, j-1}$ for indices taken modulo $n+2$. As suggested by our weave fillings, we can also consider the $\Z_{n+2}$ action of counterclockwise rotation on the set of diagonals 
	$\{D_{i, j}\}_\mathcal{T}$ of a triangulation $\mathcal{T}$ of the $(n+2)$-gon. 
	Restricting to the toric chart induced by an augmentation $\sigma$, there is a corresponding triangulation $\mathcal{T}_\sigma$ for which it can be shown that the set map $D_{i-1, j-1}\mapsto \Delta_{i, j}$ between diagonals $\{D_{i-1, j-1}\}_{\mathcal{T}_\sigma}$ of the triangulation $\mathcal{T}_\sigma$ and regular functions $\{\Delta_{i, j}\}$ is a $\Z_{n+2}$-equivariant map.\footnote{As we explain later, this indexing shift is necessary so that the combinatorial bijection between 312-avoiding permutations and triangulations yields Lagrangian fillings that induce the same toric chart inside the augmentation variety.}	The orbital structure of the action of rotation on triangulations will therefore appear as a crucial ingredient in the proofs of Theorems \ref{thm: orbit size}. 
	
	The number of orbits of the set of triangulations of the $(n+2)$-gon under the action of counterclockwise rotation is given by the formula
	
	$$\frac{C_{n}}{n+2}+\frac{C_{n/2}}{2}+\frac{2C_{(n-1)/3}}{3} $$

	where, as previously, the terms with $C_{n/2}$ and $C_{n/3}$ only appear if the indices are integers. These terms correspond, respectively, to triangulations with no rotational symmetry, rotational symmetry by $\pi$, and rotational symmetry by $\frac{2\pi}{3}$. No other rotational symmetry of a triangulation is possible. The orbit sizes are $n+2$, $\frac{n+2}{2}$ and $\frac{n+2}{3}$, where again the corresponding orbit size only occurs if the relevant fraction is an integer.



\subsubsection{The flip graph}
	The combinatorics of triangulations of the $(n+2)$-gon have previously appeared in constructions of $A$-type fillings. As explained in \cite{TreumannZaslow, CasalsZaslow}, Legendrian mutation, an operation for generating new fillings, corresponds to exchanging diagonals of a quadrilateral in the original triangulation to form a new triangulation. Such an exchange of diagonals is depicted in Figure \ref{fig: flip}, and we refer to it as an edge flip. See Subsection \ref{sub:clusters} for more on the cluster-algebraic interpretation of this operation in terms of cluster mutation. The  flip graph or associahedron is then defined to have vertices given by triangulations and an edge between two vertices if the triangulations are related by a single edge flip. The diameter of the flip graph was first investigated via geometric methods by Thurston Sleator and Tarjan in \cite{SleatorTarjanThurston} and later 
	combinatorially by Pournin in \cite{Pournin14}. In general, the combinatorics of the flip graph are an area of active interest, and there is no known algorithm for determining geodesics. 
	In Subsections \ref{sub:rotation} and \ref{sub:mutation}, we present a description of the K\'alm\'an loop as a sequence of edge flips in the flip graph and describe the result of a single edge flip on a 312-avoiding permutation, thus providing a characterization of the K\'alm\'an loop action as a geodesic path in the flip graph. 

		\section{Isotopies of exact Lagrangian Cobordisms}\label{section:geometric}
	

	
	In this section we prove that a pinching sequence filling $L_\sigma$ is  Hamiltonian isotopic to the weave filling  $L_{\mathcal{T}_\sigma}$ for a given 312-avoiding permutation $\sigma$. 
	We first relate the elementary cobordisms used to construct these fillings. 
	
		\begin{proposition}
		\label{lemma:local models}
	The pinching cobordism and $D_4^-$ cobordism are Hamiltonian isotopic relative to their boundaries. 
		\end{proposition}

	We prove this by giving a local model for the $D_4^-$ cobordism as a sequence of diagrams in both the front and Lagrangian projections and then describing 
	an exact Lagrangian isotopy between the two cobordisms that fixes the boundary. Since compactly supported Lagrangian isotopy is equivalent to Hamiltonian isotopy \cite[Theorem 3.6.7]{ExactLagrangian=Hamiltonian}, this implies the proposition. We then use Proposition \ref{lemma:local models} to prove Theorem \ref{thm: isotopy1} in the general case of Lagrangian fillings of $\la(\beta)$. In the specific case of $\la(A_{n-1})$ we obtain the following corollary.
	
	\begin{corollary}\label{cor: isotopyAn}
	   The pinching sequence filling $L_\sigma$ is Hamiltonian isotopic to the weave filling $L_{\mathcal{T}_\sigma}$.
	\end{corollary}
	
	The vertical weave construction we use in the proof of Corollary \ref{cor: isotopyAn} also allows us to argue that a 312-avoiding permutation yields a unique pinching sequence filling up to Hamiltonian isotopy, as we explain below. Finally, we conclude the section with a proof of Theorem \ref{thm: orbit size} (1) as a further corollary of Theorem \ref{thm: isotopy1}. 
	
\subsection{Proof of Theorem \ref{thm: isotopy1}} We begin with a proof of Proposition \ref{lemma:local models}.
	
	\begin{proof}[Proof of Proposition \ref{lemma:local models}] 
We give two local models of a $D_4^-$ cobordism, depicted in Figures \ref{fig:localmiddle} and \ref{fig:localleft} as movies in the front (top) and Lagrangian (bottom) projections. The first local model depicts the removal a Reeb chord trapped between a pair of crossings and a 0-resolution of the rightmost crossing. The second local model depicts the removal of a Reeb chord originally appearing to the left of the leftmost crossing and a 0-resolution of this crossing. This is accomplished by first applying a Legendrian isotopy to create a pair of crossings with this Reeb chord trapped between them and proceeding as in the first local model. 

The main difficulty in our comparison of these local models to the pinching cobordism is to unambiguously relate the Reeb chord removed in the $D_4^-$ cobordism to the Reeb chord removed in the pinching cobordism. This means that we must carefully manipulate the slope of the Legendrian in the front projection to ensure that no new Reeb chords are introduced throughout the process. The local models allow us to verify by inspection that no new Reeb chords appear at any point in this cobordism, as the slopes of the front projection are specified so that no new intersections appear in the Lagrangian projection. 
	
Armed with a local model for the slicing of the $D_4^-$ cobordism, we now describe an exact Lagrangian isotopy between this local model and the pinching cobordism. Starting in the front projection of $\la(\beta)$, a slicing of the pinching cobordism as defined in Subsection \ref{sub: pinching cobordism} consists of applying the Ng resolution, resolving a crossing, and then undoing the Ng resolution. Restricting to a neighborhood of a crossing allows us to describe the desired isotopy.

First, consider a contractible Reeb chord with a neighborhood resembling one of the two models shown in Figures \ref{fig:localmiddle} and \ref{fig:localleft}.
In such a neighborhood, the exact Lagrangian isotopy between the two cobordisms is visible when examining the Lagrangian projection of the local models depicted in Figures \ref{fig:localmiddle} and \ref{fig:localleft} (bottom). Indeed, after applying the Ng resolution, the only difference between these local models and the pinching cobordism in the Ng resolution is the rotating of the strand before resolving. Therefore, the movie of movies realizing the exact Lagrangian isotopy from the $D_4^-$ cobordism to the pinching cobordism consists of incrementally applying the Legendrian isotopy of the Ng resolution, rotating the crossing before pinching, and then undoing the Ng resolution.

	\begin{center}
		\begin{figure}[]{ \includegraphics[width=.95\textwidth]{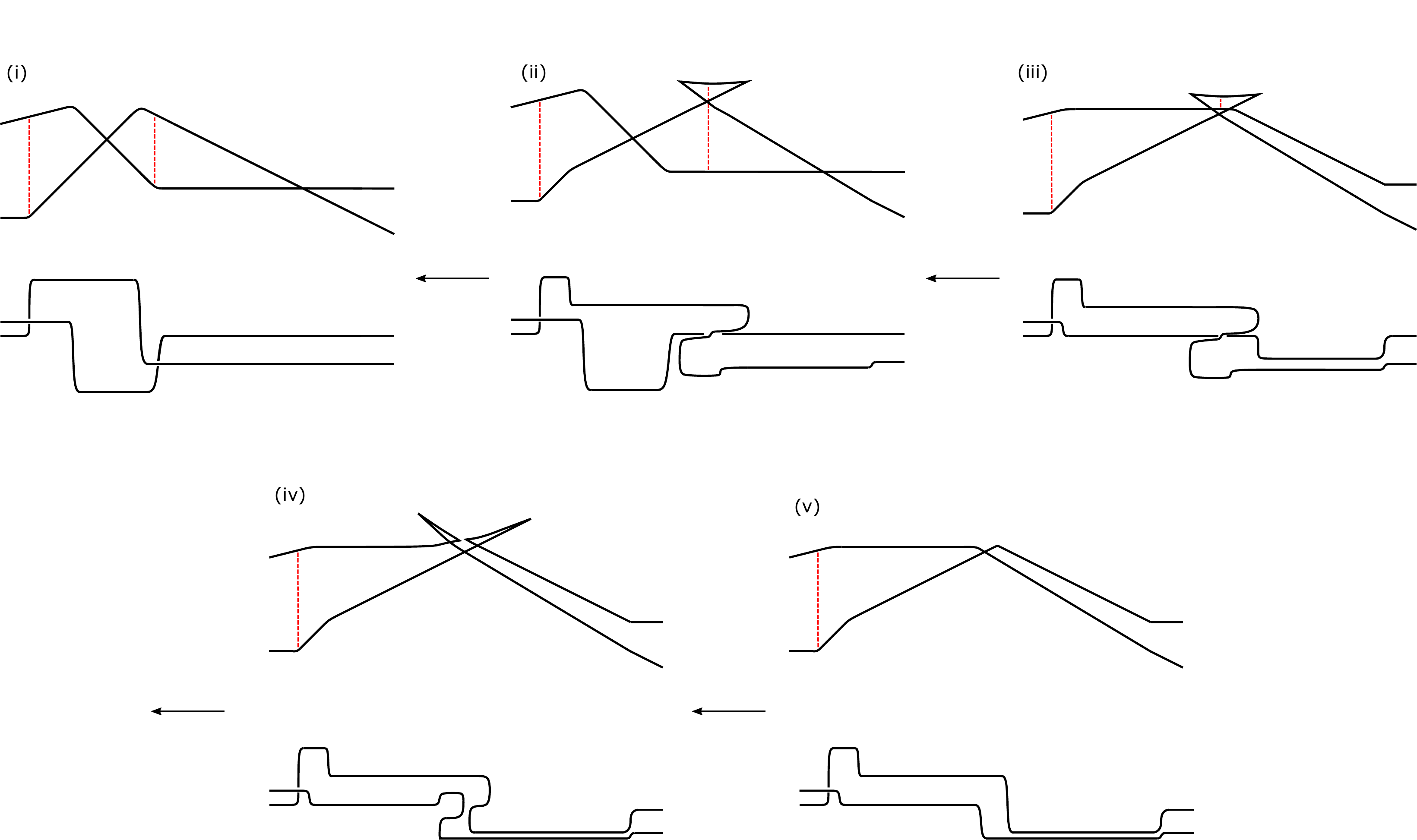}}\caption{Local model of a $D_4^-$ cobordism applied to a pair of crossings in the front (top) and Lagrangian (bottom) projections. Reeb chords are depicted by red dashed lines. The direction of the arrows indicate a  cobordism from the concave end to the convex end.}
			\label{fig:localmiddle}\end{figure}
	\end{center}

	\begin{center}
		\begin{figure}[]{ \includegraphics[width=\textwidth]{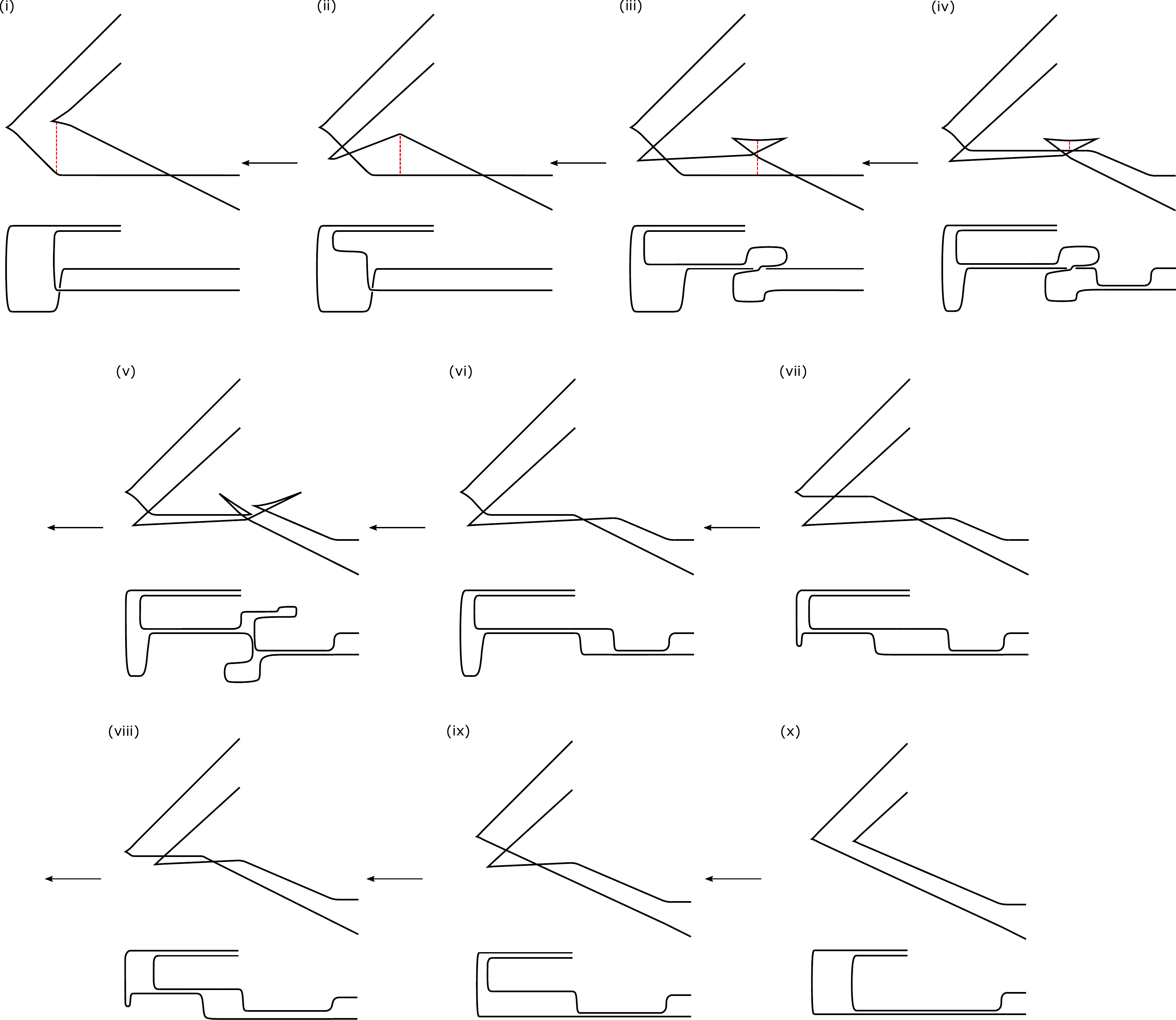}}\caption{Local model of the leftmost crossing in the front (top) and Lagrangian (bottom) projections with a single Reeb chord depicted by a red dashed line. We first apply a Reidemeister II move in order to artificially introduce an additional crossing so that there is a single Reeb chord trapped between the new crossing and the original crossing. The $D_4^-$ cobordism is performed in (iii)-(vi) and the remaining part of the cobordism undoes the Reidemeister II move without creating any new Reeb chords.}
			\label{fig:localleft}\end{figure}
	\end{center}
Now consider a Reeb chord that does not admit a neighborhood resembling one of our two local models. 
In this case, we can rotate all of the crossings that appear to the left of the Reeb chord past the cusps by performing a sequence of Reidemeister III moves 
so that we obtain a neighborhood resembling the initial figure in \ref{fig:localleft}. We then apply the local model given in Figure \ref{fig:localleft} to resolve this crossing. We can then rotate the remaining crossings back, and by analogous reasoning to above, the resulting cobordism is Hamiltonian isotopic to the pinching cobordism at $z$.\end{proof}

Now that we have established the equivalence between the pinching cobordism and the $D_4^-$ cobordism, Theorem \ref{thm: isotopy1} follows as a corollary. 

\begin{proof}[Proof of Theorem \ref{thm: isotopy1}]
    By construction, any weave filling is a decomposable Lagrangian filling made up of elementary cobordisms corresponding to Reidemeister III moves and $D_4^-$ cobordisms. By Proposition \ref{lemma:local models}, the $D_4^-$ cobordism is Hamiltonian isotopic to a pinching cobordism. Therefore, a Legendrian weave filling is Hamiltonian isotopic to a decomposable Lagrangian filling made up of Reidemeister III moves and pinching cobordisms.
\end{proof}
	
\subsection{Lagrangian fillings of $\la(A_{n})$}
	
To complete the proof of Corollary \ref{cor: isotopyAn} we argue that the clip sequence bijection defined in Subsection \ref{sub: Kalman prelims} gives a one-to-one correspondence between fillings that resolves crossings in the same order.
	
\begin{proof}[Proof of Corollary \ref{cor: isotopyAn}]
	 Let $\sigma$ be a 312-avoiding permutation indexing a pinching sequence filling $L_\sigma$ of $\la(A_{n-1})$ and consider the vertical weave corresponding to the triangulation $\mathcal{T}_\sigma$. By construction, a 0-resolution at the crossing $i$ in $\la(A_{n-1})$ corresponds to a trivalent vertex where the incident rightmost edge is labeled by $i$. By Proposition \ref{lemma:local models}, these denote Hamiltonian isotopic exact Lagrangian cobordisms applied to corresponding Reeb chords. Thus, the filling $L_\sigma$ is Hamiltonian isotopic to the weave filling dual to the triangulation $\mathcal{T}_\sigma$.\end{proof}

It is claimed without proof in \cite[Section 8.1]{EHK} that, in addition to yielding the same Floer-theoretic invariant, there is a Hamiltonian isotopy between pinching sequence fillings represented by permutations $\sigma= \dots i\, k\,\dots j\dots $ and $\sigma'=\dots k \, i \dots j\dots$ in $S_n$. This claim implies that a 312-avoiding permutation represents a unique equivalence class of Lagrangian filling up to Hamiltonian isotopy. The claim follows from Corollary \ref{cor: isotopyAn} and the lemma below.

	\begin{lemma}\label{planar isotopy}
	
	Let $(x_i, z_i), (x_j, z_j),$ and  $(x_k, z_k)$ denote the coordinates of three trivalent vertices in the 2-graph $\Gamma$ satisfying $x_i<x_j<x_k$ and $z_j<z_k<z_i$. The planar isotopy 
	between $\Gamma$ and the 2-graph $\Gamma'$ with trivalent vertices at $(x_i, z_k), (x_j, z_j),$ and $(x_k, z_i)$ 
	lifts to a compactly supported Hamiltonian isotopy of the fillings $L_\Gamma$ and $L_\Gamma'$ fixing the boundary.
	\end{lemma}

	\begin{proof}
	    By construction, the planar isotopy between $\Gamma$ and $\Gamma'$ lifts to a Legendrian isotopy between the weaves $\La(\Gamma)$ and $\La(\Gamma')$ in $J^1(D^2)$. Note that this planar isotopy can be taken to be the identity at the boundary $\partial\La(\Gamma)$. Considering the Lagrangian projection of this sequence of weaves yields a compactly supported exact Lagrangian isotopy between the Lagrangian fillings $L_\Gamma$ and $L_\Gamma'$. By \cite[Theorem 3.6.7]{ExactLagrangian=Hamiltonian}, this implies the existence of a compactly supported Hamiltonian isotopy between the two fillings. \end{proof}	  

By Corollary \ref{cor: isotopyAn}, the exact Lagrangian isotopy of the weave filling extends to pinching sequence fillings. Thus, our result implies that there are exactly a Catalan number $C_n$ of pinching sequence fillings\footnote{Note that a precise classification of fillings currently only exists for the Legendrian unknot. In general, it is not known whether every filling is constructible, i.e. can be given as a series of elementary cobordism.} of $\la(A_{n-1})$ up to Hamiltonian isotopy.

We conclude this section with a proof of the orbital structure described in Theorem \ref{thm: orbit size} (1) as a corollary of Theorem \ref{thm: isotopy1}. Namely, the orbital structure of the K\'alm\'an loop action on pinching sequence fillings of $\la(A_{n-1})$ can be obtained from the Hamiltonian isotopy between the pinching sequence filling $L_\sigma$ and weave filling $L_{\mathcal{T}_\sigma}$.

\begin{proof}[Proof of Theorem \ref{thm: orbit size} (1)]
    Let $L_\sigma$ be a filling of $\la(A_{n-1})$ and consider the Hamiltonian isotopic weave filling $L_{\mathcal{T}_\sigma}$ with corresponding 2-graph $\Gamma$ dual to the triangulation $\mathcal{T}_\sigma$.  The K\'alm\'an loop action on weave fillings is geometrically described as a cylinder rotating the entire 2-graph $\Gamma$ by $\frac{2\pi}{n+2}$ radians counterclockwise. This can be readily observed from the fact that crossings of $\la(A_{n-1})$ are represented by edges of the dual graph intersecting the boundary of the $(n+2)$-gon. Therefore, the correspondence between triangulations $\mathcal{T}_\sigma$ and weave fillings $L_{\mathcal{T}_\sigma}$ implies that the orbital structures of triangulations under the action of rotation and weave fillings under the action of the K\'alm\'an loop coincide. 
\end{proof}

Note here the appearance of $\la(A_{n-1})$ as the $(-1)$-framed closure of the braid $\sigma^{n+2}$ in the description of the weave filling. This geometrically describes why the K\'alm\'an loop action on the rainbow closure of $\sigma^n$ has order $n+2$ as an action on the $n+2$ crossings of the $(-1)$-framed closure.
	
	\section{Algebraic Proof of Theorem \ref{thm: orbit size}}\label{section:algebraic}

		In this section we prove Theorem \ref{thm: orbit size} by examining the K\'alm\'an loop action on the augmentation variety $\Aug(\la(A_{n-1}))$ of the Legendrian link $\la(A_{n-1})$. As discussed in Subsection \ref{sub: DGA prelims}, an embedded exact Lagrangian filling yields the inclusion of an algebraic torus into the augmentation variety $\Aug(\la(A_{n-1}))$. From \cite{YuPan}, we have an explicit computation of a set of coordinate functions $\{s_1, \dots s_{n-1}\}$ on an induced toric chart coming from a pinching sequence filling $L$; namely, this set of coordinates is in bijection with the relative cycles associated to the unstable manifolds of the saddle critical points 
	    for $L$. Naively, we might hope to distinguish the Hamiltonian isotopy classes of the Lagrangian fillings under the K\'alm\'an loop action by studying the associated toric charts and their $s_i$ coordinate functions. In practice, these {\it local} coordinate functions are somewhat difficult to compare under this particular action. Instead, we consider the action of the K\'alm\'an loop on the set of {\it global} regular functions $\{\Delta_{i,j}\}$ with $\Delta_{i, j}\in \Z[\Aug(\la(A_{n-1}))]$, defined in Subsection \ref{sub: braid matrices}. In fact, $\Delta_{i, j}\in \Z[z_1, \dots, z_n]$ are globally defined polynomials, which restrict to global regular functions on the augmentation variety $\Aug(\la(A_{n-1}))\subseteq \Z^{n}$. 
	
When considering the restriction of the $\Delta_{i, j}$ functions to the toric chart induced by the augmentation $\epsilon_\sigma$, Theorem \ref{thm: signs} below establishes that the correspondence between diagonals $D_{i-1, j-1}$ of the triangulation $\mathcal{T}_\sigma$ and the functions $\Delta_{i, j}$ is a $\Z_{n+2}$-equivariant map. We then show in Subsection \ref{sub:monomials} that the $\Delta_{i,j}$ functions corresponding to diagonals of a triangulation $\mathcal{T}_\sigma$ restrict to a coordinate basis of the toric chart defined by $L_\sigma$. In addition, we give an explicit formula for these coordinate functions as monomials in the $s_i$ local coordinates. It follows that the induced action on the set of augmentations $\epsilon_\sigma$ in the augmentation variety $\Aug(\la(A_{n-1}))$ is equivalent to the action of rotation on triangulations of the $(n+2)$-gon, from which we can conclude the orbital structure as given in Theorem \ref{thm: orbit size} (1). See Subsection \ref{sub:clusters} for a cluster-algebraic motivation for the $\Delta_{i, j}$ functions and triangulations of the $(n+2)$-gon.        

	\subsection{The K\'alm\'an loop action on $\{\Delta_{i,j}\}$}

Let us start by describing the action of the K\'alm\'an loop on the global regular functions $\Delta_{i, j}$ using Euler's identity for continuants. All indices in this section are modulo $n+2$. Recall that we denote by $\vartheta\in \Aut(\Z[\Aug(\la(A_{n-1}))])$ 
	the automorphism induced by the K\'alm\'an loop acting on the augmentation variety $\Aug(\la(A_{n-1}))=\{(z_1, \ldots, z_n) |X_n=0\}\subseteq \Z^n$, where $X_n\in \Z[z_1, \dots, z_n]$ is the polynomial defined by $X_n=-1+\Delta_{1, n+2}$. 
	The action of the K\'alm\'an loop on the set of global regular functions $\{\Delta_{i, j}\}$ is described in the following restatement of Theorem \ref{thm: orbit size} (2).

	\begin{theorem}\label{thm: signs} The global regular functions $\Delta_{i, j}$ in $\Z[\Aug(\la(A_{n-1}))]$ satisfy the equation
		\begin{equation}\label{eq loop action}
		    \vartheta(\Delta_{1,k+1})+(-1)^{n-1}\Delta_{k,n+2}=-\Delta_{2,k}X_n
		\end{equation}
		as global polynomials in ambient $\Z^n$ for $2< k < n+2$.
	\end{theorem}

	As a corollary, we see that the action of $\vartheta$ on the augmentation variety $\Aug(\la(A_{n-1}))$ coincides with the action of rotation on triangulations of the $n+2$-gon.
	
	\begin{corollary}\label{cor Aug rotation}
	    $\vartheta(\Delta_{i, j})=(-1)^{n}\Delta_{i-1,  j-1}$ as regular functions on $\Aug(\la(A_{n-1}))$ and the map $\Delta_{i, j}\to D_{i-1, j-1}$ is a $\Z_{n+2}$-equivariant map.
	\end{corollary}
	
	\begin{proof}[Proof of Corollary \ref{cor Aug rotation}]
	    Restricting to $\Aug(\la(A_{n-1}))=\{X_n=0\}$ causes the right hand side of Equation \ref{eq loop action} to vanish. Therefore, by Theorem \ref{thm: signs}, the K\'alm\'an loop action on the restriction of $\Delta_{i,j}$ to $\Aug(\la(A_{n-1}))$ is $\vartheta(\Delta_{i, j})=(-1)^{n}\Delta_{i-1,  j-1}$. Under rotation, the diagonal $D_{i-1, j-1}$ maps to $D_{i-2, j-2}$. It follows that the correspondence between $\Delta_{i, j}$ restricted to the toric chart induced by $\epsilon_\sigma$ and a diagonal $D_{i-1, j-1}$ of the triangulation $\mathcal{T}_\sigma$ is a $\Z_{n+2}$-equivariant map.  
	\end{proof}
		
		
We now give a proof of the behavior of the $\Delta_{i,j}$ as ambient polynomials in $\Z^n$. 
Note here the appearance of Euler's identity for continuants in the form of Equation \ref{eq Euler}.
	
	\begin{proof}[Proof of Theorem \ref{thm: signs}]

	We first rewrite the left hand side of the desired equation using the continuant recursion relation \eqref{eq:recurse} and the action of $\vartheta$. 
		
		\begin{align*}
			\vartheta(\Delta_{1, k+1})+(-1)^{n-1}\Delta_{k,n+2}
			&=\vartheta(z_1\Delta_{2, k+1}+\Delta_{3, k+1})+(-1)^{n-1}\Delta_{k,n+2}\\
			&=-\Delta_{2, n+2}\Delta_{1, k}+\Delta_{2, k}+(-1)^{n-1}\Delta_{k,n+2}.
		\end{align*}

		We substitute this expression into the left hand side of the desired equation from Theorem \ref{thm: signs} to obtain
		$$-\Delta_{2, n+2}\Delta_{1, k}+\Delta_{2, k}+ (-1)^{n-1}\Delta_{k, n+2}= -\Delta_{2,k}(\Delta_{1, n+2}-1).$$
		
		In order to verify that this equation holds, we will apply the special case of Euler's identity for continuants given in Equation \ref{eq Euler}. To do so, we distribute the right hand side and subtract $\Delta_{2, k}$ from both sides to get
		
		$$-\Delta_{2, n+2}\Delta_{1, k}+ (-1)^{n-1}\Delta_{k, n+2}= -\Delta_{2,k}\Delta_{1, n+2}.$$
		
		This expression is equivalent to  	
		$$\Delta_{1, n+2}\Delta_{2,k}-\Delta_{1,k}\Delta_{2, n+2}=(-1)^{n}\Delta_{k,n+2},$$
		which is the identity given in Equation (\ref{eq Euler}).
		Thus, we have established Theorem \ref{thm: signs} and Theorem \ref{thm: orbit size} (2).
	\end{proof}

 	\subsection{The K\'alm\'an loop action on the augmentation variety}\label{sub:monomials}
	
	We now prove that the $\Delta_{i, j}$ functions corresponding to the diagonals of the triangulation $\mathcal{T}_\sigma$ define a coordinate basis on the toric chart induced by the filling $L_\sigma$. To do so, we first show that the $\Delta_{i, j}$ functions can be written as monomials in the local $s_i$ coordinate functions defined by the augmentation $\epsilon_\sigma$. We then define a bijection between the $\Delta_{i,j}$ corresponding to the triangulation $\mathcal{T}_\sigma$ and the $s_i$ variables on the toric chart induced by $L_\sigma$. Throughout the remainder of this section, let $\sigma$ denote a 312-avoiding permutation corresponding to a pinching sequence filling and $D_{i, j}$ be a diagonal of the triangulation $\mathcal{T}_\sigma$. The goal of this subsection will be to prove the following proposition.

	\begin{proposition}\label{prop: monomials}
		
		The set of all $\Delta_{i, j}$ corresponding to the diagonals of the triangulation $\mathcal{T}_\sigma$ forms a basis for the toric chart induced by the augmentation $\epsilon_\sigma$. 
	\end{proposition}

 The technical lemma introduced below will be used to prove the first part of Proposition \ref{prop: monomials}.

		\begin{lemma}\label{lemma: Fibonacci}
			For any diagonal $D_{i-1, j-1}$ in the triangulation $\mathcal{T}_\sigma$, the image of the regular function $\Delta_{i, j}$ in the toric chart induced by the augmentation $\epsilon_{\sigma}$ is given by
		$\epsilon_\sigma(\Delta_{i,j})=s_i\dots s_{j-2}.$

		\end{lemma}

		Assuming the lemma, we first prove Proposition \ref{prop: monomials}.
		
			\begin{proof}[Proof of Proposition \ref{prop: monomials}]
        We first define a bijection $\varphi$ between the set of triangles in the triangulation $\mathcal{T}_\sigma$ and the local toric coordinates $s_1, \dots, s_{n-1}$ induced by the augmentation $\epsilon_\sigma$. Let $T$ be a triangle in $\mathcal{T}_\sigma$ with sides $D_{i-1, j-1}, D_{j-1, k-1}$ and $D_{i-1, k-1}$. We define the map $\varphi$ by

		$$\varphi(T):=(\Delta_{i, j})^{-1}(\Delta_{j, k})^{-1}\Delta_{i, k}.$$

		where we recall that $\Delta_{i, i+1}=1$ by definition. By Lemma \ref{lemma: Fibonacci}, we have $$(\Delta_{i, j})^{-1}(\Delta_{j, k})^{-1}\Delta_{i, k}= (s_i\dots s_{j-2})^{-1}(s_j\dots s_{k-2})^{-1}s_i\dots s_{k-2}=s_{j-1}.$$To see that $\varphi$ is injective, consider two triangles $T$ and $T'$ belonging to the triangulation $\mathcal{T}_\sigma$ with sides $\{ D_{i-1, j-1}, D_{j-1, k-1},D_{i-1, k-1}\}$ and $\{ D_{i'-1, j'-1}, D_{j'-1, k'-1},D_{i'-1, k'-1}\}$, respectively. Assume that $\varphi(T)=\varphi(T')$. Then $s_{j-1}=s_{j'-1}$, and therefore $j=j'.$ Since $T$ and $T'$ share a middle vertex, and belong to the same triangulation, they must be the same triangle. We can conclude immediately that $\varphi$ is bijective because it is an injective map between two sets of $n-1$ elements. Thus, 
		the set of $\Delta_{i, j}$ functions corresponding to diagonals $\mathcal{T}_\sigma$ form a coordinate basis for the toric chart induced by the augmentation $\epsilon_\sigma$.  
		
	\end{proof}

	We now give a proof of Lemma \ref{lemma: Fibonacci} by carefully examining the effect of the DGA map $\Phi$ on the braid matrices defining $\Delta_{i,j}$.
	
	\begin{proof}[Proof of Lemma \ref{lemma: Fibonacci}]
	
	        Consider $\Delta_{i,j}$ corresponding to some diagonal $D_{i-1, j-1}$ of a triangulation $\mathcal{T}_\sigma$. By definition, we have $$\Phi(\Delta_{i,j})=\left[\prod_{k=i}^{j-2}B(\Phi(z_k))\right]_{(2, 2)}.$$

Therefore, Lemma \ref{lemma: Fibonacci} is equivalent to the claim that the $(2, 2)$ entry of $\prod_{k=i}^{j-2}B(\epsilon_{\sigma}(z_k))$ is precisely $\prod_{k=i}^{j-2} s_k$. To verify this statement, we show inductively that applying $\Phi_l\circ \dots \circ \Phi_1$ yields a product of $B(z_k)$ for $k\in \{i, \dots, j-2\}\backslash \{\sigma(1)\dots, \sigma(l)\}$ with a particular collection of diagonal matrices, upper triangular and lower triangular matrices.

Define the matrices 
$$C(s):=\begin{pmatrix}1 & s\\ 0 & 1\end{pmatrix} \qquad U(s):=\begin{pmatrix} -s^{-1} & 1\\ 0 & s \end{pmatrix} \qquad L(s):=\begin{pmatrix} -s^{-1} & 0\\ 1 & s \end{pmatrix} \qquad D(s):= \begin{pmatrix} -s^{-1} & 0\\ 0 & s \end{pmatrix}$$

Denote by $A^{\sf T}$ the transpose of the matrix $A$. The following identities are immediate. 

\begin{align}
\label{eq 3} B(z\pm s)&=B(z)C(\pm s)\\
\label{eq 4}    B(z\pm s)&=C^{\sf T}(\pm s)B(z)\\
\label{eq 5}    C(s\pm t)&=C(s)C(\pm t)\\
\label{eq 6}    C(-s^{-1})B(s)C^{\sf T}(-s^{-1})&=D(s)\\
\label{eq 7}   C(-s^{-1})B(s)&=L(s)\\
\label{eq 8}    B(s)C^{\sf T}(-s^{-1})&=U(s)
\end{align}

Equipped with this set of identities, we proceed with the proof of Lemma \ref{lemma: Fibonacci}. First, we may assume that $\sigma(1), \dots, \sigma(k)$ all lie in the set $\{i,\dots j-2\}$. Indeed, for $\sigma(l)$ not in $\{i, \dots j-2 \}$, the map $\Phi_l$ is the identity on the polynomial $\Phi_{l-1}\circ \dots \circ \Phi_1(\Delta_{i,j})$. This follows from the observation that if $D_{i-1, j-1}$ is in the triangulation $\mathcal{T}_\sigma$, then $i$ appears before $i-1$ and $j-2$ appears before $j-1$ in $\sigma$ under the clip sequence bijection. 
		Therefore, $\sigma^{-1}(i)<\sigma^{-1}(i-1)$ and $\sigma^{-1}(j-2)<\sigma^{-1}(j-1)$, which implies that no elements of the set $T_\sigma^l$ appear in terms of $\Phi_{l-1}\circ \dots \circ \Phi_1(\Delta_{i,j})$.

Denote by $M_l^+$ and $M_l^-$ the maximum and minimum of the set $\{i, \dots, j-2\}\backslash \{\sigma(1),\dots, \sigma(l-1)\}.$ We claim that the result of applying $\Phi_l$ to $\Phi_{l-1}\circ \dots \circ \Phi_{1}(\Delta_{i, j})$ results in the replacement of $B(z_{\sigma(l)})$ in the product $\prod_{k=i}^{j-2} B(\Phi_{l-1}\circ\dots\circ\Phi_{1}(z_k))$ with one of three possibilities depending on $l$:

\begin{enumerate}
    \item For $\sigma(l)= M_l^-$, the map $\Phi_l$ replaces $B(z_{\sigma(l)})$ by the upper triangular matrix $U(s_{\sigma(l)})$.
    \item For $\sigma(l)= M_l^+$, the map $\Phi_l$ replaces $B(z_{\sigma(l)})$ by the lower triangular matrix $L(s_{\sigma(l)})$.
    \item For $M_l^-<\sigma(l)<M_l^+$, the map $\Phi_l$ replaces $B(z_{\sigma(l)})$ by the diagonal matrix matrix $D(s_{\sigma(l)})$. 
    
\end{enumerate}

We prove this claim by induction. For the base case, we consider the three possibilities listed above. 
\begin{enumerate}
    \item If $\sigma(1)=M_1^-=i$, then we have 
\begin{align*}
    \Phi_1(\Delta_{i, j})&= B(s_i)B(z_{i+1}-s_i^{-1})B(z_{i+2})\dots B(z_{j-2})\\
    &=B(s_i)C^{\sf T}(-s_i^{-1})B(z_{i+1})\dots B(z_{j-2})\\
    &=U(s_i)B(z_{i+1})\dots B(z_{j-2})
\end{align*}
where the second equality follows from Equation (\ref{eq 4}) and the final one from Equation (\ref{eq 8}).\\

    \item If $\sigma(1)=M_1^+=j-2$, then Equations (\ref{eq 3}) and (\ref{eq 7}) imply that $$\Phi_1(\Delta_{i,j})=B(z_i)\dots B(z_{j-3})C(-s_{j-2}^{-1})B(s_{j-2})=B(z_i)\dots B(z_{j-3})L(s_{j-2}).$$ 
    
   \item If $i<\sigma(1)<j-2$, then we apply Equations (\ref{eq 3}), (\ref{eq 4}), and (\ref{eq 6}) to $\Phi_1(\Delta_{i, j})$ to obtain
\begin{align*}
    \Phi_1(\Delta_{i,j})&=B(z_i)\dots B(z_{\sigma(1)-1})C(-s_{\sigma(1)}^{-1})B(s_{\sigma(1)})C^{\sf T}(-s_{\sigma(1)}^{-1})B(z_{\sigma(1)+1})\dots B(z_{j-2})\\
    &=B(z_i)\dots B(z_{\sigma(1)-1})D(s_{\sigma(1)})B(z_{\sigma(1)+1})\dots B(z_{j-2}).
\end{align*} 
\end{enumerate}

Assume inductively that applying the composition $\Phi_{l-1}\circ\dots \circ \Phi_1$ replaces each $B(z_{\sigma(k)})$ for $1\leq k \leq l-1$ with either $U(s_{\sigma(k)})$, $L(s_{\sigma(k)})$, or $D(s_{\sigma(k)})$ depending on whether $\sigma(k)=M_k^-$, $\sigma(k)=M_k^+$, or $M_k^-<\sigma(k)<M_k^+$ respectively. We consider the same three cases for $\Phi_l$:

\begin{enumerate}
    \item If $\sigma(l)=M_l^-$, then $T_\sigma^l$ has a single element $l'$ and by the combinatorial formula for $\Phi_l,$ we have 
    $$B(\Phi(z_{\sigma(l')}))=C^{\sf T}(\pm s_{\sigma(l)}^{-1}s_{\sigma(l)+1}^{-2}\dots s_{\sigma(l')-1}^{-2})B(z_{\sigma(l')})$$ where the sign is given by $(-1)^{|\sigma(l)-(\sigma(l')-1)|}$. 
    By the inductive hypothesis, we have that the matrices appearing between $B(\Phi(z_{\sigma(l)}))=B(s_{\sigma(l)})$ and $C^{\sf T}(s_{\sigma(l)}^{-1}s_{\sigma(l)+1}^{-2}\dots s_{\sigma(l')-1}^{-2})$ are of the form $D(s_{\sigma(l)+1})\dots D(s_{\sigma(l')-1})$. We then compute 
    \begin{align*}
        B(s_{\sigma(l)})\Bigg(\prod_{m=\sigma(l)+1}^{\sigma(l')-1}D(s_m)\Bigg)C^{\sf T}(\pm s_{\sigma(l)}^{-1}s_{\sigma(l)+1}^{-2}\dots s_{\sigma(l')-1}^{-2})&=\begin{pmatrix} \pm s_{\sigma(l)}^{-1}\dots s_{\sigma(l')-1}^{-1} & s_{\sigma(l)+1}\dots s_{\sigma(l')-1}\\ 0 & s_{\sigma(l)}\dots s_{\sigma(l')-1} \end{pmatrix}\\ 
        &= U(s_{\sigma(l)})\Bigg(\prod_{m=\sigma(l)+1}^{\sigma(l')-1}D(s_m)\Bigg),
    \end{align*}
    
    showing that applying $\Phi_l$ replaces $B(z_l)$ by $U(s_l)$.\\

    \item If $\sigma(l)=M_l^+$, then $T_\sigma^l$ again has a single element $l''$ and we have
     \begin{align*}
       C(\pm s_{\sigma(l'')+1}^{-2}\dots s_{\sigma(l)-1}^{-2}s_{\sigma(l)}^{-1}) \Bigg(\prod_{m=\sigma(l'')+1}^{\sigma(l)-1}D(s_m)\Bigg) B(s_{\sigma(l)})&=\begin{pmatrix}\mp( s_{\sigma(l'')+1}\dots s_{\sigma(l)-1})^{-1} & (s_{\sigma(l'')+1}\dots s_{\sigma(l)})^{-1} \\ 0 & (s_{\sigma(l'')+1}\dots s_{\sigma(l)-1})\end{pmatrix} \\ 
       &=\Bigg(\prod_{m=\sigma(l'')+1}^{\sigma(l)-1}D(s_m)\Bigg) L(s_{\sigma(l)}).
    \end{align*}
    
    \item Finally, if $M_l^-<\sigma(l)<M_l^+$, then $T_\sigma^l$ has two elements, denote them by $l'$ and $l''$ with $l'>l''$. Then we must consider the product

     $$ C(\pm s_{\sigma(l'')+1}^{-2}\dots s_{\sigma(l)-1}^{-2}s_{\sigma(l)}^{-1})\Bigg(\prod_{m=\sigma(l'')+1}^{\sigma(l)-1} D(s_{m})\Bigg)B(s_{\sigma(l)})\Bigg(\prod_{m=\sigma(l)+1}^{\sigma(l')-1}D(s_m)\Bigg) C^{\sf T}(\pm s_{\sigma(l)}^{-1}s_{\sigma(l)+1}^{-2}\dots s_{\sigma(l')-1}^{-2})$$
where the two signs of entries of $C$ and $C^{{\sf T}}$ need not agree. We apply our computation from the previous case and simplify
    \begin{align*}
         \Bigg(\prod_{m=\sigma(l'')+1}^{\sigma(l)-1}D(s_m)\Bigg)L(s_{\sigma(l)})\Bigg(\prod_{m=\sigma(l)+1}^{\sigma(l')-1}D(s_m)\Bigg)C^{\sf T}(\pm s_{\sigma(l)}^{-1}s_{\sigma(l)+1}^{-2}\dots s_{\sigma(l')-1}^{-2})\\
         =\Bigg(\prod_{m=\sigma(l'')+1}^{\sigma(l)-1}D(s_m)\Bigg)L(s_{\sigma(l)})\begin{pmatrix}\mp (s_{\sigma(l)+1}\dots s_{\sigma(l')-1})^{-1} & 0\\ \pm (s_{\sigma(l)}s_{\sigma(l)+1}\dots s_{\sigma(l')-1})^{-1} & s_{\sigma(l)+1}\dots s_{\sigma(l')-1}\end{pmatrix}\\
    \end{align*}
     
    This yields the product $\prod_{m=\sigma(l'')+1}^{\sigma(l')-1}D(s_m)$, as desired.

\end{enumerate}

Thus, by induction, $\Phi_l$ replaces $B(z_l)$ with an upper triangular, lower triangular, or diagonal matrix for $1< l < j-2-i$. Therefore, when we arrive at the final $\Phi_{j-2-i}$ map, we have $B_{s_{\sigma(j-2-i}})$ multiplied on the left by the product of some number of upper triangular and diagonal matrices with the corresponding $s$ variable appearing in the $(2, 2)$ entry and multiplied on the right by some number of diagonal and lower triangular matrices with the same condition. Since multiplication by a diagonal matrix preserves the property of being upper or lower triangular, the result is a product of the form $UB(s_{\sigma(j-2-i)})L$ where $U$ and $L$ are upper and lower triangular matrices. Therefore the $(2, 2)$ entry of this product is the product of the $(2, 2)$ entries of each of the factors. It follows that the $(2, 2)$ entry of $\Phi(\Delta_{i, j})$ is $s_i\dots s_{j-2}$, as desired.
\end{proof}

Together, Corollary \ref{cor Aug rotation} and Proposition \ref{prop: monomials} allow us to give an algebraic proof of Theorem \ref{thm: orbit size} (1)

\begin{proof}[Proof of Theorem \ref{thm: orbit size} (1)]
    By Proposition \ref{prop: monomials}, the set of $\Delta_{i, j}$ corresponding to a triangulation $\mathcal{T}_\sigma$ gives a basis for the toric chart induced by the augmentation $\epsilon_\sigma$. Therefore, the image of $\epsilon_\sigma$ under $\vartheta$ corresponds to the image of the set of $\Delta_{i, j}$ corresponding to $\mathcal{T}_\sigma$. By Corollary \ref{cor Aug rotation}, we know that the induced action of the K\'alm\'an loop on the augmentation variety of $\la(A_{n-1})$ is equivalent to the action of rotation on the $n+2$-gon. Therefore, the orbital structure of the action can be described as in Subsection \ref{sub: Orbital structure}. \end{proof}

	\subsection{Relation to cluster theory} \label{sub:clusters}

		The appearance of the $\Delta_{i,j}$ functions and the combinatorics of the $(n+2)$-gon is explained by a cluster structure on the augmentation variety, the existence of which was recently proven by Gao-Shen-Weng in \cite{GSW}. In brief, a cluster variety is an algebraic variety containing a set of toric charts (cluster charts) with coordinate functions (cluster variables) that transform according to a specific operation (cluster mutation) under the chart maps. See \cite{FWZ1, FWZ2} for more on cluster algebras. 

		For a Legendrian $\la$ given as the rainbow closure of a positive braid, \cite{GSW} describes a cluster structure on $\Aug(\la)$ 
		by proving a natural isomorphism to double Bott-Samelson cells. In particular, the cluster structure on $\Aug(\la(A_{n-1}))$ is a cluster algebra of $A$-type. $A$-type cluster algebras were originally defined and studied by Fomin and Zelevinksy in the context of regular functions on the affine cone of the Grassmanian $\mathcal{G}r^\times(2, n+2)$ \cite{FominZelevinsky_ClusterII}. If we consider the Pl\"ucker coordinate $P_{i,j}$ of the (ordinary) Grassmanian $Gr(2, n+2)$, then its image in the affine cone is precisely the function $\Delta_{i,j}$. The combinatorics of the relationship between cluster charts is captured by the flip graph, where a single cluster seed is given by all $\Delta_{i,j}$ corresponding to diagonals $D_{i, j}$ of a triangulation. In the context of this manuscript, \cite{GSW} implies the existence of cluster coordinates on $\Aug(\la(A_{n-1}))$ while Proposition \ref{prop: monomials} gives a precise formula.


    Also of interest in the cluster setting is the fact that the K\'alm\'an loop induces a cluster automorphism of the augmentation variety $\Aug(\la(A_{n-1}))$. Subsection \ref{sub:rotation} explicitly realizes this automorphism as a sequence of mutations. For an $A$-type cluster algebra, Assem, Schiffler, and Shramchenko showed that the cluster automorphism group is $\Z_{n+2}$ \cite{AssemSchifflerShramchenko}. Theorem \ref{thm: signs} implies that the order of the K\'alm\'an loop action on $\Aug(\la(A_{n-1}))$ is precisely $n+2$, so we immediately deduce the following corollary.
	\begin{corollary}
		The induced action of the K\'alm\'an loop on $\Aug(\la(A_{n-1}))$ is a generator of the $A$-type cluster modular group. 
	\end{corollary}

	\section{Combinatorial Characterizations}\label{section:combinatorics}

	In this section, we describe the combinatorial properties of the K\'alm\'an loop action on a pinching sequence filling $L_\sigma$ of $\la(A_{n-1})$ purely in terms of the corresponding 312-avoiding permutation $\sigma$. We first present an explicit algorithm for determining the orbit size of $L_\sigma$ from $\sigma$ in Subsection \ref{sub:orbits}. The end of the subsection includes a table where orbit sizes are computed for the case $n=4$, corresponding to triangulations of the hexagon. We then give a recipe for constructing a geodesic path in the flip graph that describes a counterclockwise rotation of the triangulation $\mathcal{T}_\sigma$. Since the weave filling $L_{\mathcal{T}_\sigma}$ is Hamiltonian isotopic to the pinching sequence filling $L_\sigma$ by Theorem \ref{thm: isotopy1}, this geodesic path describes the K\'alm\'an loop action on $L_\sigma$ as a sequence of edge flips. 
	Finally, we discuss the behavior of 312-avoiding permutations under a single edge flip in the flip graph. Together, these last two results give a combinatorial characterization of the K\'alm\'an loop action on fillings purely in terms of 312-avoiding permutations. 
	As in previous sections, all indices are computed modulo $n+2$.

		\subsection{Orbit size}\label{sub:orbits}
	To prove Theorem \ref{thm: algorithm}, we give explicit criteria in Lemmas \ref{lemma: Algorithm1} and \ref{lemma: Algorithm2} for when a filling of $\la(A_{n-1})$ has orbit size $\frac{n+2}{2}$ or $\frac{n+2}{3}$ under the action of the K\'alm\'an loop. If it does not satisfy either of these criteria, then it necessarily has orbit size $n+2$. We start by describing the permutations that arise from an orbit of size $\frac{n+2}{2}$.

	Consider some 312-avoiding permutation $\sigma\in S_{n}$. In order for the filling $L_\sigma$ to have orbit size $\frac{n+2}{2},$ the triangulation $\mathcal{T}_\sigma$ must have rotational symmetry through an angle of $\pi$. Therefore, $\mathcal{T}_\sigma$ has a diameter $D_{i, i+\frac{n+2}{2}}$ and the triangulated polygons on either side of this diameter must be mirror images. We consider the diameter as an external edge of two $(\frac{n+2}{2}+1)$-gons, one containing both vertices labeled $n+1$ and $n+2$, and the other containing at most one of them. For any 312-avoiding permutation $\sigma$ such that $\mathcal{T}_\sigma$ has a diameter $D_{i,i+\frac{n+2}{2}}$, we define the 312-avoiding permutation $\tau$ in $S_{\frac{n+2}{2}}$ corresponding to half of the triangulation $\mathcal{T}_\sigma$ as follows.

	\begin{definition}
	  The permutation $\tau$ in the letters $i+1, i+1, \dots, i+\frac{n+2}{2}-1$ is the 312-avoiding permutation obtained from applying the clip sequence bijection to the triangulation $\mathcal{T}_\tau$ of the $(\frac{n+2}{2}+1)$-gon containing at most one of the vertices labeled by $n+1$ and $n+2$.
	\end{definition}
	
	We can always unambiguously identify $\tau$ from the permutation $\sigma$.
\begin{lemma}\label{lemma tau}
Let $\sigma\in S_n$ be any 312-avoiding permutation such that $\mathcal{T}_\sigma$ has rotational symmetry through an angle of $\pi$. The permutation $\tau$ is the first subword of $\sigma$ of length $\frac{n}{2}$ letters forming a subinterval of the integers $\{1, \dots, n\}$. 
\end{lemma}
	
	\begin{proof}
	    Let $\sigma\in S_n$ correspond to a triangulation $\mathcal{T}_\sigma$ with rotational symmetry through an angle of $\pi$. Under the clip sequence bijection, there may be letters of $\sigma$ that appear before $\tau$ 
	    Therefore, to identify $\tau$ as a subword of $\sigma$ we search for the first 312-avoiding permutation of length $\frac{n}{2}$ that appears in $\sigma$. A diameter $D_{i, i+\frac{n+2}{2}}$ forces the condition that any letters appearing before $\tau$ will be less than $i$, so that even if $i$ appears directly after $\tau$, there is no ambiguity in identifying $\tau$.

    Explicitly, we identify $\tau$ by first checking if the set $\{\sigma(1), \dots \sigma(\frac{n}{2})\}$ of the first $\frac{n}{2}$ letters of $\sigma$ is equal to a subinterval of the integers $\{1, \dots, n\}$ of length $\frac{n}{2}$. If not, we check the $\{\sigma(2), \dots, \sigma(\frac{n}{2}+1)\}$. We continue in this way until we have either identified the subword $\tau$ or exhausted all possibilities. If no such subword exists, then $\sigma$ does not have the assumed rotational symmetry.
	\end{proof}


    We now state a preparatory lemma regarding details of the clip sequence bijection that may give some insight into the structure of the orbit size algorithm below. We consider the most general case where $\mathcal{T}_\tau$ is a subtriangulation of $\mathcal{T}_\sigma$ with vertices $i, \dots, i+k$ for $i+k\leq n+1$.

    \begin{lemma}\label{lemma: maxtriangle}
    Let $k\in \N$ satisfy $j< i+k \leq n+1$. The 312-avoiding permutation $\tau$ ends in the letter $j$ if and only if the subtriangulation $\mathcal{T}_\tau$ contains the triangle labeled by vertices $i, j,$ and $i+k$. 
    In this case, all letters taking values strictly between $i$ and $j$ appear before any other letters in $\tau$.
    \end{lemma}

    \begin{proof}

        The first claim follows from the definition of the clip sequence bijection because the diagonal $D_{i, i+k}$ must appear in the final triangle remaining after removing the previous $n-1$ vertices. Therefore, $j$ is the final letter of $\tau$, if and only if it is also the third vertex of this triangle. 
        
        The second claim follows by similar reasoning to the case of the diameter, as the existence of the diagonal $D_{i, j}$ implies that there must be some ear $D_{l, l+2}$ with $i<l< j-2$. Therefore, $l+1$ appears before $j$ and we can repeat this argument for the subtriangulation of $\mathcal{T}_\tau$ obtained by removing the vertex $l+1$.
    \end{proof}
    
	
	We now give explicit criteria for determining whether the filling $L_\sigma$ has orbit size $\frac{n+2}{2}$ solely in terms of $\sigma$.

	 \begin{lemma}\label{lemma: Algorithm1}

	 The following algorithm detects whether a 312-avoiding permutation $\sigma$ in $S_n$ yields a filling $L_\sigma$ of orbit size $\frac{n+2}{2}$ under the action of the K\'alm\'an loop.
	 \end{lemma}

	\begin{enumerate}

	\item Identify $\tau$ from $\sigma$ as in the proof of lemma \ref{lemma tau}.
	
	\item Define $\sigma'$ to be an empty string and set $\tau'=\tau$. Find the smallest  $j$ for which $j\geq \frac{n+2}{2}$ and some $k>j$ appears before $j$ in $\tau.$ For the first such $k$ appearing in $\tau'$, append $k-\frac{n+2}{2}$ to $\sigma'$, remove $k$ from $\tau'$ and repeat until no such letters remain in $\tau'$. Append $\tau$ to $\sigma'$.
		
		\item While $\tau'$ ends in the largest (resp. smallest) number  remaining in $\tau'$ not equal to $\frac{n+2}{2}-1$ (resp. $\frac{n+2}{2}$), then append the next largest (resp. next smallest) number in $\{1,\dots,  n\}\backslash \sigma'$ less than the smallest number (resp. greater than the largest number) 
		of $\tau'$ to $\sigma'$ and delete the final number of $\tau'$.  
		\item If $\tau'$ does not end in its largest or smallest remaining number, then add $\frac{n+2}{2}$ to all numbers less than the final number and append to $\sigma'$ in the order they appear. Delete the corresponding numbers from $\tau'.$
		\item Now $\tau'$ ends in its smallest remaining number, so return to Step (3) and repeat until only one number remains in $\tau'$. The final number of $\sigma'$ is then determined by the unique number remaining in $\{1, \dots, n\}\backslash\sigma'$.
		\item $\sigma$ has orbit size $\frac{n+2}{2}$ if it is equal to $\sigma'$.
	\end{enumerate}

	\begin{ex}
Consider the 312-avoiding permutation $\sigma = 1\, 5\, 4\, 3\, 6\, 2.$ We can identify $\tau=5\, 4\, 3$ as the first length 3 subword appearing in $\sigma$ and the diameter of the triangulation $\mathcal{T}_\sigma$ is therefore $D_{2, 6}$. Applying the above algorithm to $\tau$, we see that Step (2) yields $\sigma'=1\, 5\, 4\, 3$ because 5 precedes 4. Then 3 is the smallest number appearing in $\tau$, so we append 6 to $\sigma'$. Finally, we append 2, to get $\sigma=\sigma'$, indicating that the filling labeled by $\sigma$ has orbit size 4 under the K\'alm\'an loop. 
\hfill $\Box$
	\end{ex}

\begin{proof}

	Let $\sigma\in S_n$ be a 312-avoiding permutation with orbit size $\frac{n+2}{2}$. Denote the diameter of $\mathcal{T}_\sigma$ as $D_{i, i+\frac{n+2}{2}}$ for some $1\leq i\leq \frac{n+2}{2}-1$ and the permutation corresponding to the triangulation of the $(\frac{n+2}{2}+1)$-gon given by the vertices $i, \dots, i+\frac{n+2}{2}$ by $\tau.$  We will show that the algorithm detects when  the triangulation $\mathcal{T}_\sigma$ is obtained from the triangulation $\mathcal{T}_\tau$ by gluing $\mathcal{T}_\tau$ to a rotation of $\mathcal{T}_\tau$ by $\pi$ along the diagonal $D_{i, i+\frac{n+2}{2}}$. The lemma then follows from the observation that $\tau$ is uniquely determined from $\sigma$.

		Under the clip sequence bijection, we delete the smallest vertex with no incident diagonal at each step and append the label to the permutation. Therefore, any letter $k$ of $\sigma$ appearing before $\tau$ is less than $i$. Moreover, 
		any diagonal $D_{j, k}$ or $D_{k, j}$ (should it exist) incident to the vertex $k$  has endpoint $j$ in the set $\{n+2, 1, \dots, i\}$. Therefore any triangle with vertices $j, k, $and $l$ with $j, k, l$ given in clockwise order must also have $j, l \in \{n+2, 1, \dots, i\}$. The rotational symmetry of $\mathcal{T}_\sigma$ implies that the triangle with vertices $j+\frac{n+2}{2}, k+\frac{n+2}{2}, l+\frac{n+2}{2}$ appears in $\mathcal{T}_\tau$. It follows from Lemma \ref{lemma: maxtriangle} that in $\tau$ the letter $k+\frac{n+2}{2}$ precedes $j$ for some $j\geq \frac{n+2}{2}$ and that all such $k$ appear before $\tau$ in $\sigma$. Therefore, Step (2) produces all letters of $\sigma$ that appear before $\tau$.
		

  To determine the letters following $\tau$ in $\sigma,$ we first consider the case where one of the diameter vertices, $i$ or $i+\frac{n+2}{2}$, has no incident diagonals with endpoint taking values in the set of vertices labeled by letters appearing after $\tau$ in $\sigma$. 
  If this is the case, then the appropriate diameter vertex label immediately follows $\tau$ in $\sigma$ under the clip sequence bijection. We also observe that when $i$ (respectively, $i+\frac{n+2}{2})$ is such a vertex, then there is a triangle in $\mathcal{T}_\sigma$ with vertices $i, i-1,$ and $i+\frac{n+2}{2}$  (resp. $i, i+\frac{n+2}{2}, i+\frac{n+2}{2}+1)$. Therefore, the rotational symmetry of $\mathcal{T}_\sigma$ implies that we have a triangle with vertices $i, i+\frac{n+2}{2}-1$ and $i+\frac{n+2}{2}$ (resp. $i, i+1$, and $i+\frac{n+2}{2}$) in $\mathcal{T}_\tau$. By Lemma \ref{lemma: maxtriangle}, the vertex $i+\frac{n+2}{2}-1$ (resp. $i+1$) appears as the final letter in $\tau$. The vertex $i-1$ (resp. $i+\frac{n+2}{2}+1)$ then appears immediately following $\tau$. The same reasoning applies if we replace the diameter $D_{i, i+\frac{n+2}{2}}$ with the diagonal $D_{i-1, i+\frac{n+2}{2}}$, $D_{i, i+\frac{n+2}{2}+1}$, or any such longest remaining diagonal arising under the clip sequence bijection in this way, so long as $n+1$ or $n+2$ do not appear as endpoints of this diagonal. 

  If both diameter vertices have diagonals incident to them with endpoints in the remaining vertices, then the letter following $\tau$ under the clip sequence bijection labels the smallest vertex greater than $i+\frac{n+2}{2}$ with no incident diagonals. By previous reasoning, we know that the diameter is one side of a triangle with vertices $i, k, i+\frac{n+2}{2}$ in $\mathcal{T}_\sigma$. The rotational symmetry of $\mathcal{T}_\sigma$ implies that the triangle labeled by $i, k-\frac{n+2}{2}, i+\frac{n+2}{2}$ appears in $\mathcal{T}_\tau$. It follows from Lemma \ref{lemma: maxtriangle} that $k$ appears as the final letter of $\tau$ and any letter $j$ with $j<k$ appearing before $k$ in $\tau$

  
  	
	This process continues until we have eliminated all numbers from $\tau$ except for either $n+1$ or $n+2$. This unambiguously determines the final number of our permutation. By construction, we have shown that the above algorithm yields the 312-avoiding permutation $\sigma$ with $\mathcal{T}_\sigma$ constructed by gluing a rotated copy of $\mathcal{T}_\tau$ to $\mathcal{T}_\tau$. \end{proof}
	

	We now consider the case of a 312-avoiding permutation $\sigma$ with orbit size $\frac{n+2}{3}$. In order to exhibit the appropriate rotational symmetry, the triangulation $\mathcal{T}_\sigma$ must have a central triangle labeled by vertices $i, i+\frac{n+2}{3}, i+\frac{2(n+2)}{3}$, dividing the triangulation up into three identical triangulations of $(\frac{n+2}{3}+1)-$gons. Two of these polygons do not contain the pair of vertices $n+1$ and $n+2$, so a permutation $\sigma$ with $\mathcal{T}_\sigma$ having rotational symmetry through an angle of $\frac{2\pi}{3}$ must have two subwords $\tau_1$ and $\tau_2$ of length $\frac{n+2}{3}-1$ that differ by $\frac{n+2}{3}$ and are immediately followed by $i+\frac{n+2}{3}$. We determine the third subword from $\tau_1$ using the same reasoning as in the $\frac{n+2}{2}$ orbit size case. 
	
	\begin{lemma}\label{lemma: Algorithm2}
	The following algorithm detects whether a 312-avoiding permutation $\sigma$ in $S_n$ yields a filling $L_\sigma$ of orbit size $\frac{n+2}{3}$ under the action of the K\'alm\'an loop. 

	 \end{lemma}

	\begin{enumerate}
		\item Determine $\tau_1$ by finding the first subword of length $\frac{n+2}{3}-1$ in $\sigma$ with letters $i, \dots, i+\frac{n+2}{3}-1$ for some $i.$ If no such $\tau_1$ exists, then $\sigma$ does not have orbit size $\frac{n+2}{3}$.
		\item Set $\sigma'$ to be the empty word. For any numbers greater than $\frac{n+2}{3}$ that appear after $\frac{n+2}{3}$ or some other number greater than $\frac{n+2}{3}$, add $\frac{2(n+2)}{3}$ (mod $n+2$) to them and append the result to $\sigma'$. Append $\tau_1$ to $\sigma'$. Add $\frac{n+2}{3}$ to each entry of $\tau_1$ to get $\tau_2$ and append to $\sigma'$. Append $i+\frac{n+2}{3}$ to $\sigma'$. Delete the corresponding number from $\tau_1$.
		\item So long as $\tau_1$ ends in the largest (resp. smallest) number  remaining in $\tau_1$ not equal to $\frac{n+2}{3}-1$ (resp. $\frac{n+2}{3}$), then append the next largest (resp. next smallest) number of $\{1,\dots,  n\}\backslash \sigma'$ less than the smallest number (resp. greater than the largest number) 
		of $\tau_1$ to $\sigma'$ and delete the final number in $\tau_1$.  
		\item If $\tau_1$ does not yet end in the largest or smallest remaining number, add $\frac{n+2}{3}$ to all numbers less than the final number and append. Delete the corresponding numbers from $\tau_1.$
		\item Now $\tau_1$ ends in the smallest remaining number, so return to Step (3) and continue until one number remains in $\tau_1$. The final number of $\sigma'$ is then determined by the unique number appearing in $\{1, \dots, n\}\backslash \sigma'$.
		
		\item $\sigma$ has orbit size $\frac{n+2}{3}$ if it is equal to $\sigma'$. 
	\end{enumerate}

	\begin{ex}
	We can identify $\sigma=2\,1\,5\,4\,3\,6\,7$ as a permutation with orbit size $\frac{7+2}{3}=3$ using the above algorithm. First identify $\tau_1$ as the first length 2 subword with two consecutive letters, i.e., $\tau_1= 2\,1\in S_2$. Then $\tau_2=5\,4$ and the string $2\,1\,5\,4\,3$ must appear in $\sigma$ in order for it to have orbit size 3. We can also determine that no letters appear before  $\tau_1$ because 1 and 2 already appear in our word. Since $\tau_1$ ends with the smallest letter of the triangulation $\mathcal{T}_{\tau_1},$ we append 6. The final remaining number is 7, so we see that $\sigma'=\sigma$ and therefore $\sigma$ has orbit size 3. \hfill $\Box$
	\end{ex}

	We conclude this subsection with a table of orbit sizes of pinching sequence fillings of $\la(A_3)$, i.e. the case $n=4$. 
	\begin{center}
		\begin{tabular}{ |c|c|c| }
			\hline
			\text{Permutation}  & \text{Orbit Size}\\\hline 
			
			1 2 3 4  &  6\\
			\hline
			1 2 4 3 & 3\\
			\hline
			1 3 2 4 & 2\\
			\hline
			1 3 4 2 & 6 \\
			\hline
			1 4 3 2 & 3\\
			\hline
			2 1 3 4 &  3\\
			\hline
			2 1 4 3 & 6\\
			\hline
			2 3 1 4 & 3\\
			\hline
			2 3 4 1 & 6\\
			\hline
			2 4 3 1 & 2\\
			\hline
			3 2 1 4 & 6\\
			\hline
			3 2 4 1 & 3\\
			\hline
			3 4 2 1 & 3\\
			\hline
			4 3 2 1 & 6\\
			\hline
			
		\end{tabular}
	\end{center}

	\subsection{Rotations of Triangulations}\label{sub:rotation}
	
	In this subsection, we describe a counterclockwise rotation of the $(n+2)$-gon through an angle of $\frac{2\pi}{n+2}$ as a geodesic path in the flip graph. We refer to any triangle with edges made up solely of diagonals $D_{i, i+j}$ for $j\geq 2$ as an internal triangle, and we denote the number of internal triangles in a triangulation $\mathcal{T}_\sigma$ by $t_\sigma$.

	We will say that a diagonal $D_{i, j}$ is (counter)clockwise to another diagonal $D_{i, j'}$ if the vertex $j$ is (counter)clockwise to $j'$. Similarly, $D_{i, j}$ is (counter)clockwise to $D_{i', j}$ if $i$ is (counter)clockwise to $i'$. Given a triangulation $\mathcal{T}_\sigma$, the following algorithm describes a sequence of $n-1+t_\sigma$ edge flips that produce a rotation of $\mathcal{T}_\sigma$ by $\frac{2\pi}{n+2}$ radians in the counterclockwise direction.

		\begin{enumerate}

			\item For any diagonals $D_{i,j}$ with no incident diagonal counterclockwise to it, perform an edge flip at $D_{i, j}$ to get $D_{i-1, j-1}$. Continue to flip any such diagonals not previously flipped until no such diagonals remain.

			\item Choose an internal triangle $T$ with a diagonal $D_{i,j}$ not previously flipped  and 
	    	admitting no incident diagonal $D_{i', j}$ counterclockwise to it. 
		
	       Perform an edge flip at $D_{i,j}$ and then flip any diagonals not previously flipped that have no incident counterclockwise diagonals. 
	        
	        \item If a diagonal $D_{i', j'}$ 
	        of $T$ does have incident counterclockwise diagonals, then perform an edge flip at the counterclockwise-most of these incident diagonals. Flip any diagonals not previously flipped that now admit no incident counterclockwise diagonals.
	        
	        
	        \item Repeat Step (3) until no diagonals counterclockwise to $D_{i', j'}$ remain. Perform an edge flip at $D_{i', j'}$. Once the second and third diagonals of $T$ have been flipped, perform an edge flip at the initial diagonal previously belonging to $T$.
	        
	        \item Repeat Steps (3) and (4) starting with the remaining diagonals in the triangle corresponding to the counterclockwise diagonal flipped in Step (3). Continue until all possible diagonals have been flipped at.
	        
		\end{enumerate}
		
\begin{center}
		\begin{figure}[h!]{ \includegraphics[width=\textwidth]{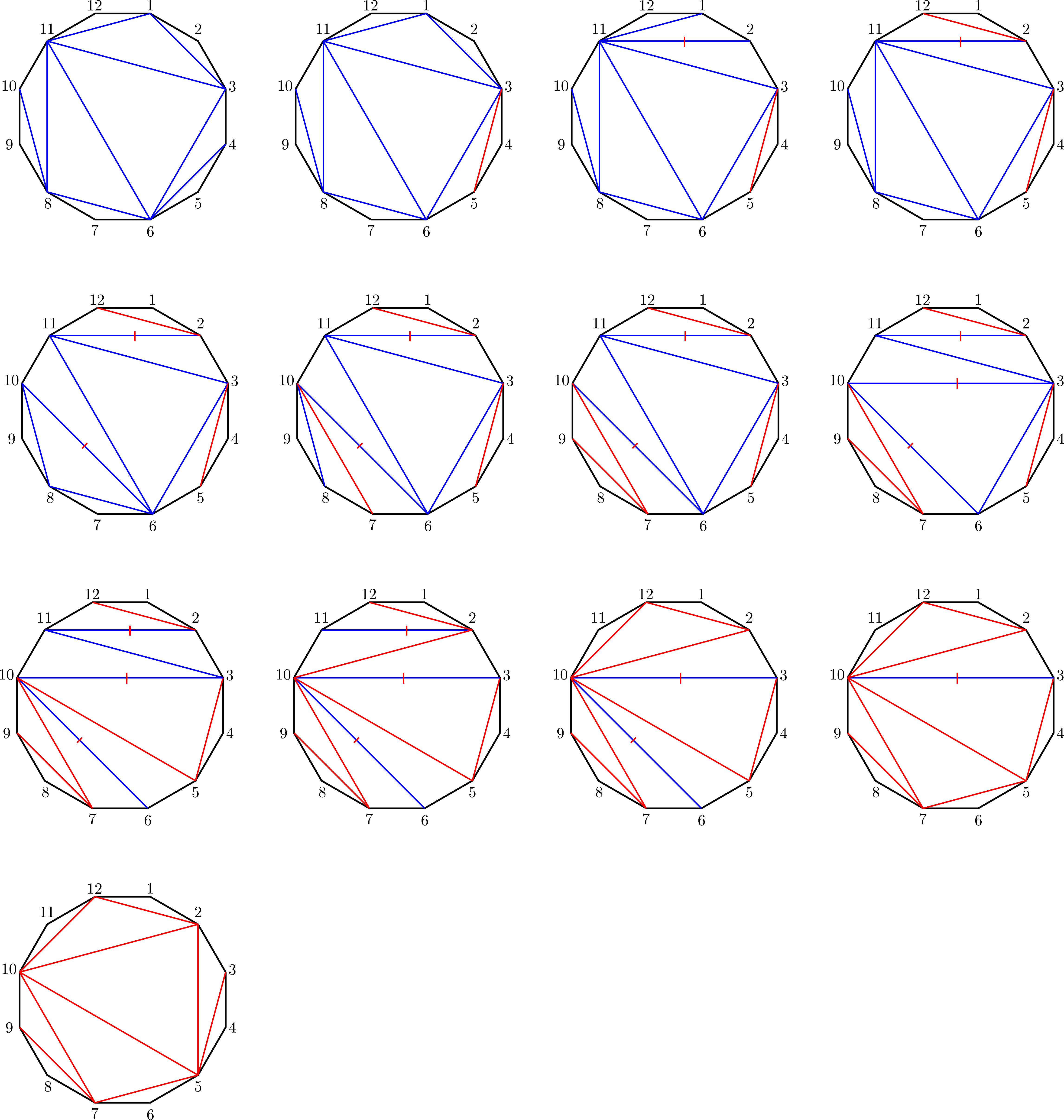}}\caption{	Counterclockwise rotation of a triangulation of the 12-gon by 10-1+3=12 edge flips. The red diagonals are diagonals of the rotated triangulation, while the blue diagonals with a red mark are diagonals that are the result of a previous edge flip but are not diagonals of the rotated triangulation.}
			\label{fig: rotation}\end{figure}
	\end{center}

	\begin{ex}
	If the triangulation $\mathcal{T}_\sigma$ only contains diagonals of the form $D_{i, j_1}, \dots, D_{i, j_{n-1}}$, then the instructions above reduce to simply performing edge flips in reverse order of indexing, starting with $D_{i, j_{n-1}}$ and ending with $D_{i, j_1}$. See Figure \ref{fig: rotation} for a more involved example with three internal triangles. 
	\end{ex}

\begin{theorem}\label{thm: rotation}


The number of edge flips required to realize a counterclockwise rotation of a triangulation $\mathcal{T}_\sigma$ of the $(n+2)$-gon by $\frac{2\pi}{n+2}$ is $n-1+t_\sigma$. The above instructions describe a sequence of $n-1+t_\sigma$ edge flips realizing such a rotation.

	\end{theorem}


	\begin{proof}

		We first argue that the number of flips needed to rotate a triangulation is at least $n-1+t_\sigma$. Since no diagonal of our original triangulation is a diagonal of our rotated triangulation, a rotation of the triangulation $\mathcal{T}_\sigma$ requires at least $n-1$ edge flips, i.e. as many edge flips as diagonals of $\mathcal{T}_\sigma$. However, in an internal triangle, it is not possible to apply a single edge flip to any of the three sides (or any other diagonal) so that the result is a side of the rotated triangle, or indeed any diagonal of the rotated triangulation. This is because each of the three sides prevents the side immediately counterclockwise to it from rotating in a counterclockwise direction. If none of the internal triangles share a side, then the claim follows. Otherwise, we argue that any two triangles sharing an edge still require at least two extra edge flips to rotate. The only possible way we could have fewer is if we could perform an edge flip at the shared side and then rotate the two triangles with a single edge flip of each of the remaining sides. However, if we apply an edge flip at the shared side, then the remaining sides of the two triangles prevent the opposite pair from achieving the desired rotation. Therefore, we must have at least $n-1+t_\sigma$ edge flips for a rotation of $\frac{2\pi}{n+2}$.

		The algorithm given above describes a path in the flip graph of length $n-1+t_\sigma$ since we have two edge flips for a single diagonal in each internal triangle and one for every other diagonal. It remains to show that the result is a rotation of the initial triangulation $\mathcal{T}_\sigma$. In Step (1), an edge flip at a diagonal $D_{i, j}$ results in the diagonal $D_{i-1, j-1}$ precisely because there are no diagonals counterclockwise to it 
		and therefore $D_{i, j}$ is a diagonal of the quadrilateral with sides $D_{i, j-1}, D_{j-1, j}, D_{i-1, j}, D_{i-1, i}$. 
		It follows that each edge flip in Step (1) results in a diagonal of the rotated triangulation. If the triangulation $\mathcal{T}_\sigma$ has no internal triangles, then applying Step (1) to each of the $n-1$ diagonals results in the desired rotation.
		
		Suppose that $\mathcal{T}_\sigma$ has at least one internal triangle. In Step (2), an edge flip at the diagonal $D_{i, j}$ in an internal triangle $\{D_{i, j}, D_{j, k}, D_{i, k}\}$ with no diagonal counterclockwise to it, results in the diagonal $D_{j-1, k}.$ 
 		Once the remaining diagonals of the triangle have no incident counterclockwise diagonals, Step (4) applies an edge flip to them so that $D_{i, k}$ becomes $D_{i-1, j-1}$ and $D_{j, k}$ becomes $D_{j-1, k-1}$. 
 		Step (4) then flips $D_{j-1, k}$ to $D_{i-1, k-1}$. Crucially, the order of edge flips ensures that during Steps (2)-(4), we strictly decrease the number of counterclockwise incident diagonals to $D_{j, k}$ and $D_{i, k}$ at each step. After rotating our initial triangle, we can continue this process with the next internal triangle. 

		It remains to show that in Step (2), a diagonal $D_{i, j}$ of an internal triangle with no incident counterclockwise diagonals $D_{i', j}$ always exists 
		If $D_{i,j}$ has an incident counterclockwise diagonal not belonging to an internal triangle, then Step (1) will apply an edge flip at such a diagonal so that it is no longer counterclockwise to $D_{i, j}$. If $D_{i, j}$ has an incident counterclockwise diagonal that belongs to an internal triangle, then there is some counterclockwise-most diagonal $D_{i, j'}$ also belonging to an internal triangle. Note that an edge flip at $D_{i, j'}$ removes one of the diagonals counterclockwise to $D_{i, j}$, so we can repeat this argument until we have performed an edge flip at all such diagonals.
	\end{proof}

	\begin{remark}
	For triangulations that allow for a choice of ordering edge flips, it follows from a theorem of Pournin's \cite[Theorem 2]{Pournin14} that naively proceeding with any of the equivalent options will still yield a geodesic. We can reinterpret this in the cluster algebraic setting as the fact that distant mutations commute. In this context, any geodesic path gives the mutations describing the cluster automorphism induced by the K\'alm\'an loop. \hfill $\Box$
	\end{remark}

		\subsection{Edge flips in terms of permutations}\label{sub:mutation}
	
	In this subsection, we describe an edge flip at a diagonal $D_{j,l}$ of the triangulation $\mathcal{T}_\sigma$ in terms of the 312-avoiding permutation $\sigma$. 
	
	Let $\sigma\in S_n$ be a 312-avoiding permutation with corresponding triangulation given by the clip sequence bijection. Consider a quadrilateral with sides $D_{i, j}, D_{j, k}, D_{k, l}$ and $D_{i, l}$ appearing in the triangulation $\mathcal{T}_\sigma$. Figure \ref{fig:Squares} depicts this quadrilateral with two possible diagonals, $D_{i, k}$ and $D_{j, l}$ separating it into two triangles. An edge flip at one of these diagonals yields the other.   
	
	\begin{center}
		\begin{figure}[h!]{ \includegraphics[width=.6\textwidth]{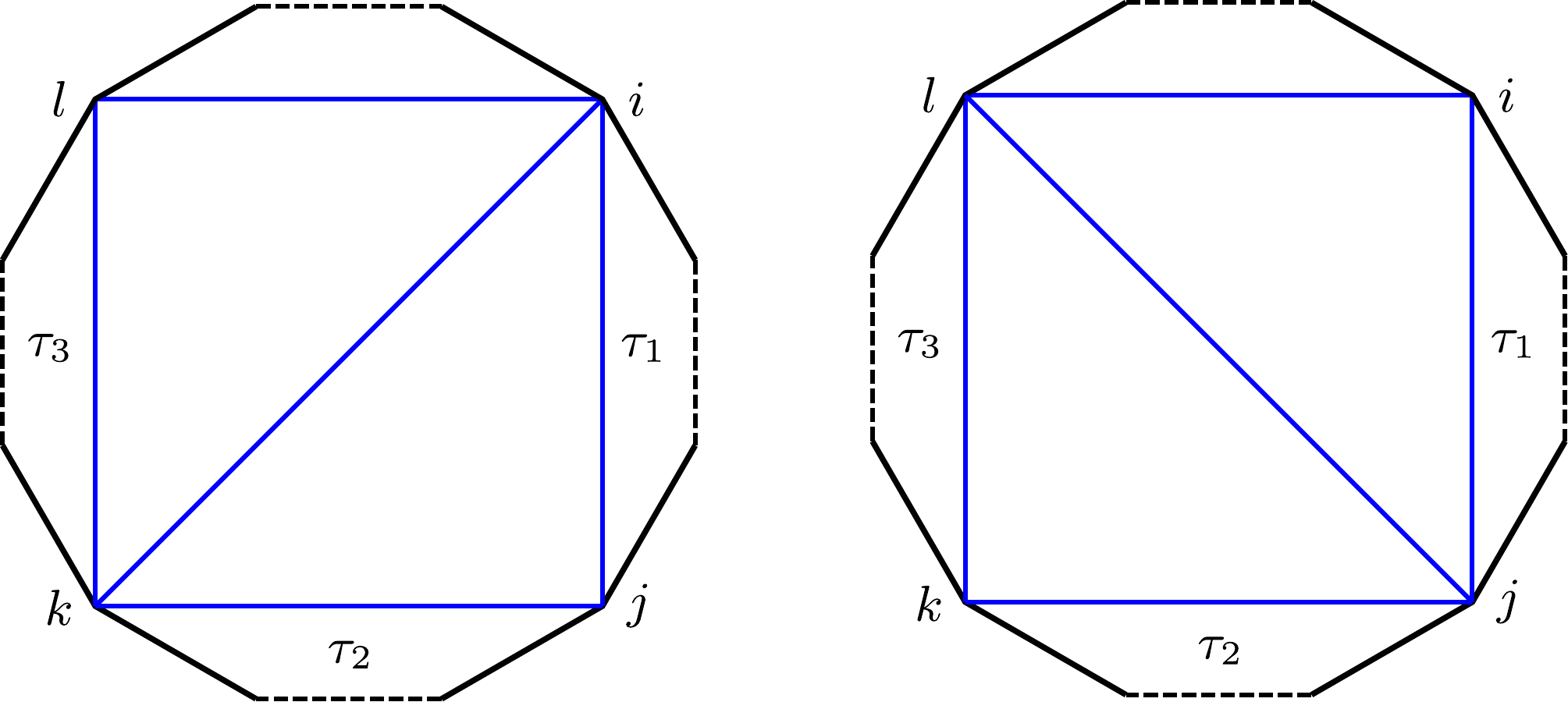}}\caption{Schematic of an edge flip depicting the triangulation $\mathcal{T}_\sigma$ (left) and the result of applying an edge flip to $\mathcal{T}_\sigma$ at $D_{i, k}$ (right). The dotted lines represent arbitrarily many edges of the $(n+2)$-gon and the indices are chosen so that either $1\leq i< j<k<l\leq n+1$ or $j<k<l=n+1, i=n+2$. The labels $\tau_1, \tau_2,$ and $\tau_3$ represent subwords of $\sigma$ corresponding to different section of $\mathcal{T}_\sigma$. If any of the edges of the quadrilateral lie on the $(n+2)$-gon, then we consider the corresponding $\tau_i$ to be the empty word.}
			\label{fig:Squares}\end{figure}
	\end{center}

	As in the orbit size algorithm, we can determine the structure of $\sigma$ based on the existence of the edges of the quadrilateral. Specifically, $\sigma$ admits subwords $\tau_1, \tau_2, $  and $\tau_3$, where the subword $\tau_1$ contains the letters $i+1, \dots, j-1$, the subword $\tau_2$ contains letters $j+1, k-1$, and the subword $\tau_3$ contains letters $k+1, \dots l-1$.  From this construction, we can deduce the effect on $\sigma$ of a single edge flip at $D_{i, j}$. 
	
	\begin{theorem}\label{thm: mutation} 
	Given a triangulation $\mathcal{T}_\sigma$ containing a quadrilateral with diagonal $D_{i, k}$, the 312-avoiding permutation $\sigma$ is of the form $\dots \tau_1\tau_2j\tau_3k\dots$. An edge flip at the diagonal $D_{j, l}$ yields a permutation of the form $\dots\tau_1 \tau_2\tau_3 kj \dots$. 
	\end{theorem}

	\begin{proof}
    The theorem follows from the observation that each $\tau_i$ must contain at least one ear -- a triangle of with edges $D_{i, i+1}, D_{i, i+2}, D_{i+1, i+2}$ -- of the triangulation $\mathcal{T}_\sigma$. Therefore, under the clip sequence bijection, the word $\tau_i$ appears before $\tau_j$ if $i<j$. Moreover, the vertex labels $j, k$ appear only after the two quadrants immediately adjacent to the vertex have been deleted under the clip sequence process. Thus, the two 312-avoiding permutations corresponding to the triangulation $\mathcal{T}_\sigma$ and the triangulation resulting from applying an edge flip are precisely of the form described.
	\end{proof}

	\begin{ex}
	Consider the permutation $\sigma=1\, 5\, 4\, 3\, 6\, 2$. If we wish to apply an edge flip to the diagonal $D_{2,6}$, then we can identify the vertex labels of the relevant quadrilateral as $i=2, j=3, k=6,$ and $l=7$. This immediately tells us that $\tau_1$ and $\tau_3$ are both empty and $\tau_2$ is the subword $5 \, 4$. Therefore, Theorem \ref{thm: mutation} above implies that we simply interchange $j$ and $k$ to get the resulting permutation $\mu(\sigma)=1\, 5\, 4\, 6\, 3\, 2.$ See Figure \ref{fig: flip} for the triangulations $\mathcal{T}_\sigma$ and the triangulation resulting from the edge flip. \hfill $\Box$
	\end{ex}

		\begin{center}
		\begin{figure}[h!]{ \includegraphics[width=.6\textwidth]{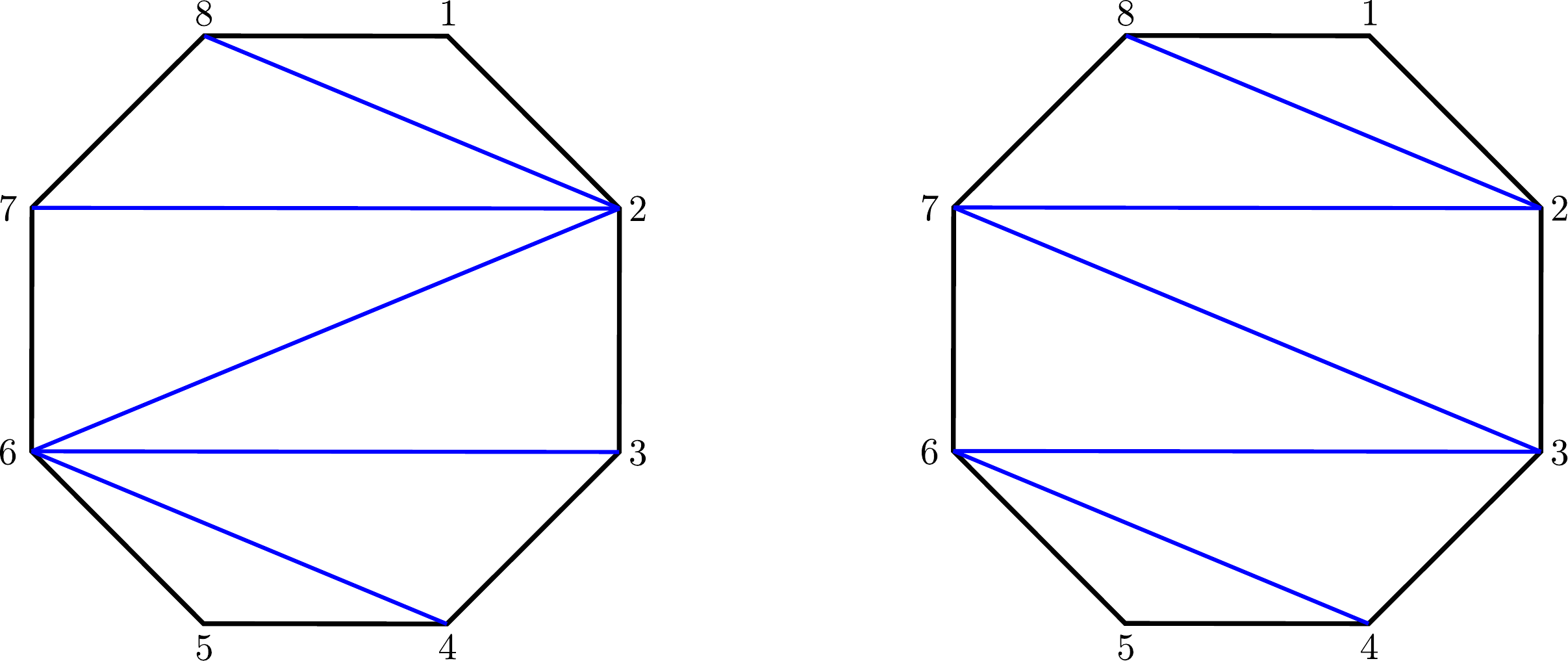}}\caption{An edge flip at the diagonal $D_{2,6}$ in the triangulation $\mathcal{T}_{1\, 5\, 4\, 3\, 6\, 2}$ yields the permutation $1\, 5\, 4\, 6\, 3\, 2$.}
			\label{fig: flip}\end{figure}
	\end{center}
Together with Theorem \ref{thm: rotation}, the above computation gives an explicit combinatorial construction of K\'alm\'an loop in terms of geodesics paths of the flip graph and the corresponding behavior of 312-avoiding permutations.

	\bibliographystyle{alpha}
	\bibliography{AnRotations}

\end{document}